\documentclass[a4paper]{amsart}

\usepackage{amsfonts}
\usepackage{amssymb}
\usepackage{amsthm}
\usepackage{amsmath}
\usepackage[all]{xy}

\SelectTips{cm}{}

\theoremstyle{plain}
\newtheorem{thm}{Theorem}[section]
\newtheorem{prop}[thm]{Proposition}
\newtheorem{cor}[thm]{Corollary}
\newtheorem{lem}[thm]{Lemma}

\theoremstyle{definition}
\newtheorem{defn}[thm]{Definition}

\theoremstyle{remark}
\newtheorem{rem}[thm]{Remark}
\newtheorem{expl}[thm]{Example}

\numberwithin{equation}{section}

\newcommand{\sN}{\mathcal{N}}
\newcommand{\sS}{\mathcal{S}}

\newcommand{\bt}{\mathbf{t}}
\newcommand{\bs}{\mathbf{s}}
\newcommand{\br}{\mathbf{r}}

\newcommand{\wq}{\widetilde{q}}
\newcommand{\wX}{\widetilde{X}}
\newcommand{\wL}{\widetilde{L}}
\newcommand{\wrho}{\widetilde{\rho}}

\newcommand{\hy}{\widehat{y}}
\newcommand{\hrho}{\widehat{\rho}}

\newcommand{\Gh}{\widehat{G}}

\newcommand{\CC}{\mathbb{C}}

\newcommand{\NN}{\mathbb{N}}

\newcommand{\QQ}{\mathbb{Q}}

\newcommand{\ZZ}{\mathbb{Z}}

\newcommand{\id}{\textup{id}}

\newcommand{\Hom}{\textup{Hom}}

\newcommand{\ra}{\rightarrow}
\newcommand{\xra}{\xrightarrow}
\newcommand{\lra}{\longrightarrow}
\newcommand{\co}{\colon\!}
\newcommand{\blank}{\textup{---}}

\newcommand{\divides}{\; | \;}

\newcommand{\G}{\textup{G}}
\newcommand{\TOP}{\textup{TOP}}
\newcommand{\Gsign}{\textup{G-sign}}
\renewcommand{\mod}{\textup{mod }}
\newcommand{\res}{\textup{res}}
\renewcommand{\max}{\textup{max}}
\newcommand{\smin}{\smallsetminus}

\newcommand{\RhG}{R_{\widehat G}}
\newcommand{\RhGp}{R_{\widehat G}^+}
\newcommand{\RhGm}{R_{\widehat G}^-}

\newcommand{\pr}{\textup{pr}}

\title[On Fake lens spaces]{On fake lens spaces with the fundamental
group \\ of order a power of $2$}
\author{Tibor Macko, Christian Wegner}



\address{Mathematisches Institut \\ Universit\"at M\"unster \\
Einsteinstra{\ss}e 62 \\ M\"unster, D-48149 \\ Germany \\ and
Matematick\'y \'Ustav SAV \\ \v Stef\'anikova 49 \\ Bratislava,
SK-81473 \\ Slovakia} \email{macko@math.uni-muenster.de}
\address{Mathematisches Institut \\ Universit\"at M\"unster \\
Einsteinstra{\ss}e 62 \\ M\"unster, D-48149 \\ Germany}
\email{c.wegner@uni-muenster.de}

\thanks{Both authors are supported by SFB 478 Geometrische Strukturen in der Mathematik, M\"unster.}

\begin{document}

\maketitle

\begin{abstract}
We present a classification of fake lens spaces of dimension $\geq
5$ which have as fundamental group the cyclic group of order $N = 2^K$, in that we extend
the results of Wall and others in the case $N = 2$.
\end{abstract}


\section*{Introduction}


A {\it fake lens space} is the orbit space of a free action of a
finite cyclic group $G$ on a sphere $S^{2d-1}$. It is a
generalization of the notion of a {\it lens space} which is the
orbit space of a free action which comes from a unitary
representation. The classification of lens spaces is a classical
topic in algebraic topology and algebraic $K$-theory well explained
for example in \cite{Milnor(1966)}. For the classification of fake
lens spaces in dimension $\geq 5$ methods of surgery theory are
especially suitable. The classification of fake lens spaces with $G$
of order $N = 2$ or $N$ odd was obtained and published in the books
\cite{Wall(1999)}, \cite{LdM(1971)}. Since then, the problem
remained open for $N \neq 2$ even. In this paper we address the
classification for $N = 2^K$.

One reason why the classification for all $N$ was not finished in
\cite{Wall(1999)} seems to be that the so-called $L$-groups $L^s_n
(G)$ for $G = \ZZ_N$ were unknown for $N$ even. This is not the case
anymore, see for example \cite{Hambleton-Taylor(2000)}. Using this
additional information and the general methods of Wall from
\cite[chapter 14]{Wall(1999)} we reduce the classification question
to a problem in the representation theory of $G$. The main
contribution of the present paper is that we develop calculational
methods for solving this rather complicated problem and we obtain
the solution for $N = 2^K$. 

The classification of fake lens spaces up to simple homotopy
equivalence for all $N \in \NN$ via Reidemeister torsion is
described in \cite[chapter 14E]{Wall(1999)}. The desired
homeomorphism classification within a simple homotopy type can be
formulated in terms of the {\it simple structure set} $\sS^s(X)$ of
a closed $n$-manifold $X$. An element of $\sS^s(X)$ is represented
by a simple homotopy equivalence $f \co M \ra X$ from a closed
$n$-manifold $M$. Two such $f \co M \ra X$, $f' \co M' \ra X$ are
equivalent if there exists a homeomorphism $h \co M \ra M'$ such
that $f' \circ h \simeq f$. The simple structure set $\sS^s(X)$ is a
priori just a pointed set with the base point $\id \co X \ra X$.
However, it can also be endowed with a preferred structure (in some
sense) of an abelian group (see \cite[chapter 18]{Ranicki(1992)}).

In general the simple structure set of an $n$-manifold for $n \geq
5$ can be determined by examining the surgery exact sequence which
is recalled below as (\ref{ses}). Besides determining $\sS^s (X)$ it
is also important to find invariants that distinguish its elements.
In fact the calculation of $\sS^s (X)$ is often conducted by
combining the surgery exact sequence with such invariants. This is
the case also for fake lens spaces. Here it follows from the
calculations of Wall in \cite[chapter 14E]{Wall(1999)} that the
simple structure set is detected by the $\rho$-invariant of
\cite{Atiyah-Singer-III(1968)} and \cite[chapter 14B]{Wall(1999)},
and by the so-called normal invariants. Our main Theorem
\ref{main-thm} calculates the simple structure set explicitly when
$N=2^K$. This should be seen as an improvement of the detecting
result of Wall. This interpretation follows from Corollary
\ref{main-thm-2} which says that if $N = 2^K$ there is another
collection of invariants which yields a one-to-one correspondence.
The collection contains the $\rho$-invariant as before, but the
other invariants are new. They depend on a certain choice and
certainly a geometric interpretation would still be
desirable.

Another issue that arises is the action of the group of simple
homotopy equivalences $G^s (X)$ of a manifold $X$ on $\sS^s (X)$ by
post-composition. The orbits of this action are the homeomorphism
types of manifolds simple homotopy equivalent to $X$ rather than
homeomorphism types of manifolds equipped with a simple homotopy equivalence to $X$. Hence it is an interesting question to describe the action. Following Wall one can slightly modify the question and study the polarized homeomorphism types of polarized fake lens spaces. These are fake lens spaces equipped with a choice of orientation and a choice of a generator of the fundamental group. Corollary \ref{main-thm-3} describes this classification.


\section{Statement of results}


\begin{defn}
A  {\it fake lens space} $L^{2d-1}(\alpha)$ is a manifold obtained
as the orbit space of a free action $\alpha$ of the group $G =
\ZZ_N$ on $S^{2d-1}$.
\end{defn}

\noindent The fake lens space $L^{2d-1}(\alpha)$ is a
$(2d-1)$-dimensional manifold with $\pi_1 (L^{2d-1}(\alpha)) \cong G
= \ZZ_N$ and universal cover $S^{2d-1}$. The main theorem in this
paper is:

\begin{thm}\label{main-thm}
Let $L^{2d-1}(\alpha)$ be a fake lens space with $\pi_1
(L^{2d-1}(\alpha)) \cong \ZZ_N$ where $N = 2^K$ and $d \geq 3$. Then
we have
\begin{equation} \label{sset-calc}
\sS^s (L^{2d-1}(\alpha)) \cong \bar \Sigma \oplus \bar T \cong \bar
\Sigma \oplus \bigoplus_{i=1}^{c} \ZZ_2 \oplus \bigoplus_{i=1}^{c}
\ZZ_{2^{\min\{K,2i\}}}
\end{equation}
where $\bar \Sigma$ is a free abelian group of rank $N/2-1$ if
$d=2e+1$ and $N/2$ if $d=2e$ and $c = \lfloor (d-1)/2 \rfloor$.
\end{thm}

The isomorphism (\ref{sset-calc}) has an interpretation in terms of
known geometric invariants. These are the {\it reduced $\rho$-invariant} and the {\it normal invariants} from
surgery theory as follows.

The reduced $\rho$-invariant is a homomorphism
\[
\wrho \co \sS^s (L^{2d-1}(\alpha)) \lra \QQ \RhG^{(-1)^d}
\]
where the target is the underlying abelian group of the
$(-1)^d$-eigenspace of the rationalized complex representation ring
of $G$ modulo the ideal generated by the regular representation. The
group $\bar \Sigma$ is defined as the image of $\wrho$.

The normal invariant is a homomorphism $\eta \co \sS^s (L^{2d-1}(\alpha)) \ra \sN (L^{2d-1}(\alpha))$ with the target the group of normal invariants from surgery theory, which is easily calculable. The
reduced $\rho$-invariant induces the homomorphism
\[
[\wrho] \co \widetilde \sN (L^{2d-1}(\alpha)) \lra \QQ
\RhG^{(-1)^d}/4 \cdot \RhG^{(-1)^d}.
\]
Here the source is the subgroup of $\sN (L^{2d-1}(\alpha))$ given by
the image of $\eta$ and in the target we have the quotient group
modulo the subgroup of elements in the $(-1)^d$-eigenspace of the
representation ring which are divisible by $4$. We use formulas of
Wall to show relations between the invariants $\wrho$ and $\eta$ in
Proposition \ref{rho-formula-lens-sp}.

The group $\bar T$ is defined as the kernel of $[\wrho]$. In the
proof of Proposition \ref{how-to-split-alpha-k} we describe a map
$\lambda \co \bar T \ra \sS^s (L^{2d-1}(\alpha))$ which fits into a
short exact sequence
\[
0 \lra \bar T \xra{\lambda} \sS^s (L^{2d-1}(\alpha)) \xra{\wrho}
\bar \Sigma \lra 0.
\]
Since $\bar \Sigma$ is a free abelian group the sequence splits and
we obtain the isomorphism of Theorem \ref{main-thm}. We denote the projection map on $\bar T$ by $\br \co \sS^s
(L^{2d-1}(\alpha)) \ra \bar T$. Our main technical result is the calculation of $\bar T$ in Theorems \ref{T-2}, \ref{T-N}. It tells us that it is a direct sum of copies
of $2$-primary cyclic groups which are indexed by $1 \leq i \leq 2c
$. We denote the projection on the $i$-th summand by $\br_{2i}$.
Putting these considerations together we obtain the following
corollary.

\begin{cor} \label{main-thm-2}
Let $L^{2d-1}(\alpha)$ be a fake lens space with $\pi_1
(L^{2d-1}(\alpha)) \cong \ZZ_N$ where $N = 2^K$ and $d \geq 3$.
There exists a collection of invariants
\[
\br_{4i} \co \sS^s (L^{2d-1}(\alpha)) \lra \ZZ_{2^{\min\{K,2i\}}}
\quad \textup{and} \quad \br_{4i-2} \co \sS^s (L^{2d-1}(\alpha))
\lra \ZZ_2
\]
where $1 \leq i \leq c = \lfloor (d-1)/2 \rfloor$ which together
with the $\wrho$-invariant induce a one-to-one correspondence
between elements $a \in \sS^s(L^{2d-1}(\alpha))$ and
\begin{enumerate}
\item $\wrho (a) \in \bar \Sigma \subset \QQ \RhG^{(-1)^d}$
\item $\br_{2i} (a) \in \ZZ_2, \ZZ_{2^{\min\{K,2i\}}}$.
\end{enumerate}
\end{cor}

The invariants $\br_{4i-2}$ are the normal invariants $\bt_{4i-2}$
from \cite[chapter 14E]{Wall(1999)} and the invariants $\br_{4i}$ are related to the invariants $\bt_{4i}$ from  \cite[chapter 14E]{Wall(1999)}, but they are not the same. The invariants $\bt_{2i}$ can in principle be calculated using characteristic classes (see \cite{Morgan-Sullivan(1974)}) and for the lens spaces this has been done in \cite{Young(1998)}, but the calculation does not include fake lens spaces. Admittedly, a similar `calculational' description of the invariants $\br_{4i}$ would be desirable. We suspect that this might be related to calculations of codimension $2$ surgery obstructions to `desuspension' of fake lens spaces (see subsection \ref{subsec:join}). These are elements of the so-called $LS$-groups of \cite[chapter 11]{Wall(1999)}, \cite[chapter 7]{Ranicki(1981)}. We plan to address this aspect in a future work.

The above results are about classification within a simple
homotopy type. As stated in the introduction the simple homotopy
types of fake lens spaces can be distinguished by the Reidemeister
torsion which is a unit in $\QQ R_G$, the rational group ring of $G$
modulo the ideal generated by the norm element. 

To obtain classification of fake lens spaces rather than
classification of elements of the simple structure set we follow Wall and work with polarized fake lens spaces, see Definition \ref{def-pol-lens-spc}. The simple homotopy type of a polarized lens space is given uniquely by its Reidemeister torsion as described in Proposition \ref{prop-simple-htpy-class}. That means that for two polarized fake lens spaces $L^{2d-1} (\alpha)$ and $L^{2d-1} (\beta)$ with the fundamental group $G$ there is a simple homotopy equivalence $f_{\alpha,\beta} \co L^{2d-1}(\alpha) \ra L^{2d-1} (\beta)$ of polarized fake lens spaces unique up to homotopy if and only if the Reidemeister torsions of $L^{2d-1}(\alpha)$ and $L^{2d-1}(\beta)$ coincide. This $f_{\alpha,\beta}$ gives us an element of the simple structure set $\sS^s (L^{2d-1}(\beta))$. We can formulate the classification as follows:

\begin{cor} \label{main-thm-3}
Let $L^{2d-1} (\alpha)$ and $L^{2d-1} (\beta)$ be polarized lens
spaces with the fundamental group $G = \ZZ_N$, where $N = 2^K$ and
$d \geq 3$. There exists a polarized homeomorphism between $L^{2d-1}
(\alpha)$ and $L^{2d-1} (\beta)$ if and only if
\begin{enumerate}
 \item $\Delta (L^{2d-1}(\alpha)) = \Delta (L^{2d-1} (\beta))$,
 \item $\rho (L^{2d-1}(\alpha)) = \rho (L^{2d-1} (\beta))$,
 \item $\br_{2i} (f_{\alpha,\beta}) = 0$ for all $i$.
\end{enumerate}
\end{cor}

The paper is organized as follows. In section
\ref{sec:htpy-classification} we briefly recall the simple homotopy
classification of fake lens spaces. In section \ref{sec:ses} we
recall the general machinery of surgery theory and we describe the
known terms in the surgery exact sequence of the fake lens spaces.
In section \ref{sec:rho-invariant} we recall the definition and
properties of the $\rho$-invariant. Finally, in section
\ref{sec:calculation} we prove our main technical result which is
the calculation of the group $\bar T$. Sections
\ref{sec:htpy-classification}, \ref{sec:ses} and most of section
\ref{sec:rho-invariant} is a summary of known results. Our contribution is
concentrated in a part of section \ref{sec:rho-invariant} and in the
last section \ref{sec:calculation}.

We thank Diarmuid Crowley, Ian Hambleton and Andrew Ranicki for
useful comments.


\section{Homotopy classification} \label{sec:htpy-classification}


In this section we briefly recall without proofs the statements of
the homotopy and simple homotopy classification of fake lens spaces
from \cite[chapter 14E]{Wall(1999)}. Apart from definitions only
Corollary \ref{lens-spaces-give-all-htpy-types} is of importance for
the rest of the paper.

We start by introducing some notation for {\it lens spaces} which
are a special sort of fake lens spaces. Let $N \in \NN$, $\bar k =
(k_1, \ldots k_d)$, where $k_i \in \ZZ$ are such that $(k_i,N)=1$.
When $G = \ZZ_N$ define a representation $\alpha_{\bar k}$ of $G$ on
$\CC^d$ by $(z_1 \ldots , z_n) \mapsto (z_1 e^{2\pi i k_1/N},
\ldots, z_n e^{2\pi i k_d/N})$. Any free representation of $G$ on a
$d$-dimensional complex vector space is isomorphic to some
$\alpha_{\bar k}$. The representation $\alpha_{\bar k}$ induces a
free action of $G$ on $S^{2d-1}$ which we still denote $\alpha_{\bar
k}$.

\begin{defn} \label{defn-lens-space}
A {\it lens space} $L^{2d-1}(\alpha_{\bar k})$ is a manifold
obtained as the orbit pace of a free action $\alpha_{\bar k}$ of the
group $G = \ZZ_N$ on $S^{2d-1}$ for some $\bar k = (k_1, \ldots
k_d)$ as above.\footnote{In the notation of \cite[chapter
14E]{Wall(1999)} we have $L(\alpha_{\bar k}) =
L(N,k_1,\ldots,k_n)$.}
\end{defn}

The lens space $L^{2d-1}(\alpha_{\bar k})$ is a $(2d-1)$-dimensional
manifold with $\pi_1 (L^{2d-1}(\alpha_{\bar k})) \cong \ZZ_N$. Its
universal cover is $S^{2d-1}$, hence $\pi_i (L^{2d-1}(\alpha_{\bar
k})) \cong \pi_i (S^{2d-1})$ for $i \geq 2$. There exists a
convenient choice of a CW-structure for $L^{2d-1}(\alpha_{\bar k})$
with one cell $e_i$ in each dimension $0 \leq i \leq 2d-1$.
Moreover, we have $H_i (L^{2d-1}(\alpha_{\bar k})) \cong \ZZ$ when
$i = 0,2d-1$, $H_i (L^{2d-1}(\alpha_{\bar k})) \cong \ZZ_N$ when $0
< i < 2d-1$ is odd and $H_i (L^{2d-1}(\alpha_{\bar k})) \cong 0$
when $i \neq 0$ is even.

The classification of the lens spaces up to homotopy equivalence and
simple homotopy equivalence is presented for example in
\cite{Milnor(1966)}. The simple homotopy classification is stated in
terms of Reidemeister torsion which is a unit in $\QQ R_G$. This
ring is defined as $\QQ R_G = \QQ \otimes R_G$ with $R_G = \ZZ G /
\langle Z \rangle$ where $\ZZ G$ be the group ring of $G$ and
$\langle Z \rangle$ is the ideal generated by the norm element $Z$
of $G$. We also suppose that a generator $T$ of $G$ is chosen. There
is also an augmentation map $\varepsilon' \co R_G \ra \ZZ_N$
\cite[page 214]{Wall(1999)}. The homotopy classification is stated
in terms of a certain unit in $\ZZ_N$. These invariants also suffice
for the homotopy and simple homotopy classification of finite
CW-complexes $L$ with $\pi_1 (L) \cong \ZZ_N$ and with the universal
cover homotopy equivalent to $S^{2d-1}$ of which fake lens spaces
are obviously a special case. It is convenient to make the following
definition.
\begin{defn} \label{def-pol-lens-spc}
A {\it polarization} of a CW-complex $L$ as above is a pair $(T,e)$
where $T$ is a choice of a generator of $\pi_1 (L)$ and $e$ is a
choice of a homotopy equivalence $e \co \wL \ra S^{2d-1}$.
\end{defn}
\noindent Denote further by $L^{2d-1}(\alpha_k)$ the lens space
$L^{2d-1}(\alpha_{\bar k})$ with $\bar k = (1,\ldots,1,k)$. By
$L^i(\alpha_1)$ is denoted the $i$-skeleton of the lens space
$L^{2d-1}(\alpha_1)$. If $i$ is odd this is a lens space, if $i$ is
even this is a CW-complex obtained by attaching an $i$-cell to the
lens space of dimension $i-1$.
\begin{prop}{\cite[Theorem 14E.3]{Wall(1999)}}
\label{prop-simple-htpy-class}
Let $L$ be a finite CW-complex as above polarized by $(T,e)$. Then
there exists a simple homotopy equivalence
\[
h \co L \lra L^{2d-2}(\alpha_1) \cup_\phi e^{2d-1}
\]
preserving the polarization. It is unique up to homotopy and the
action of $G$. The chain complex differential on the right hand side
is given by $\partial_{2d-1} e^{2d-1} = e_{2d-2} (T-1) U$ for some
$U \in \ZZ G$ which maps to a unit $u \in R_G$. Furthermore, the
complex $L$ is a Poincar\'e complex.
\begin{enumerate}
\item The polarized homotopy types of such $L$ are in
one-to-one correspondence with the units in $\ZZ_N$. The
correspondence is given by $\varepsilon' (u) \in \ZZ_N$.
\item The polarized simple homotopy types of such $L$ are in
one-to-one correspondence with the units in $R_G$. The
correspondence is given by $u \in R_G$.
\end{enumerate}
\end{prop}
\noindent The existence of a fake lens space in the homotopy type of
such $L$ is addressed in \cite[Theorem 14E.4]{Wall(1999)}. Since the
units $\varepsilon' (u) \in \ZZ_N$ are exhausted by the lens spaces
$L^{2d-1}(\alpha_k)$ we obtain the following corollary.
\begin{cor} \label{lens-spaces-give-all-htpy-types}
For any fake lens space $L^{2d-1}(\alpha)$ there exists $k \in \NN$
and a homotopy equivalence
\[
h \co L^{2d-1}(\alpha) \lra L^{2d-1}(\alpha_k).
\]
\end{cor}


\section{The surgery exact sequence} \label{sec:ses}


We proceed to the homeomorphism classification within a simple
homotopy type. This is the standard task of surgery theory whose
main tool is the surgery exact sequence computing the structure set
$\sS^s(X)$ for a given $n$-manifold $X$ with $n \geq 5$:
\begin{equation} \label{ses}
\cdots \ra \sN_\partial (X \times I) \xra{\theta} L^s_{n+1} (G)
\xra{\partial} \sS^s (X) \xra{\eta} \sN(X) \xra{\theta} L^s_n (G),
\end{equation}
where $G = \pi_1 (X)$. The other terms in the sequence are reviewed
below. We note that, since $\sS^s (X)$ is a priori only a pointed
set, the `exactness' is to be understood as described in
\cite[chapter 10]{Wall(1999)} or \cite[chapter 5]{Lueck(2001)}.
However, the sequence can also be made into an exact sequence of
abelian groups by the identification with the algebraic surgery
exact sequence of Ranicki as explained in \cite[chapter
18]{Ranicki(1992)}. We will make use of this structure since it
makes certain statements and proofs easier. However, our results can
be also formulated without this identification, in a less neat way
though.

By $\sN(X)$ in (\ref{ses}) is denoted the set of {\it normal
invariants} of $X$. An element of $\sN (X)$ is represented by a {\it
degree one normal map} $(f,b) \co M \ra X$ which consists of a map
$f \co M \ra X$ of oriented closed $n$-manifolds of degree $1$ and a
stable bundle map $b \co \nu_M \ra \xi$ from the stable normal
bundle of $M$ to a stable topological reduction $\xi$ of the stable
Spivak normal fibration of $X$. Two such degree one normal maps
$(f,b) \co M \ra X$, $(f',b') \co M' \ra X$ are equivalent in
$\sN(X)$ if there exists a degree one normal map $(F,B) \co (W,M,M')
\ra (X \times I, X \times 0,X \times 1)$ of manifolds with boundary
which restricts on the two ends to $(f,b)$, $(f',b')$ respectively.
Again this is a priori a set, with a base point $(\id,\id) \co X \ra
X$. However, the Pontrjagin-Thom construction gives a bijection
\begin{equation} \label{pont-thom}
\sN(X) \xra{\cong} [X;\G/\TOP]
\end{equation}
where $[\blank,\blank]$ denotes the homotopy classes of maps and
$\G/\TOP$ is the classifying space for topological reductions of
spherical fibrations. The $H$-space structure on $\G/\TOP$ coming
from Sullivan characteristic variety theorem \cite[chapter
4]{Madsen-Milgram(1979)} (also called `disjoint union $H$-space
structure' in \cite{Ranicki(2008)}) makes $\sN (X)$ into an abelian
group. This $H$-space structure extends to an infinite loop space
structure which expresses $\sN(X)$ via localization in terms of
familiar cohomology theories.
\begin{thm}[\cite{Madsen-Milgram(1979)}] \label{g/top}
There are compatible homotopy equivalences
\begin{align*}
\mathrm{G}/\mathrm{TOP}_{(2)} & \simeq  \Pi_{i \geq 1}
K(\mathbb{Z}_{(2)},4i) \times K(\mathbb{Z}_2,4i-2), \\
\mathrm{G}/\mathrm{TOP}_{(\mathrm{odd})} & \simeq  \mathrm{BO}_{(\mathrm{odd})},\\
\mathrm{G}/\mathrm{TOP}_{(0)} & \simeq  \mathrm{BO}_{(0)} \simeq
\Pi_{i \geq 1} K(\mathbb{Q},4i).
\end{align*}
\end{thm}

\begin{cor} \label{ni}
If $X$ is rationally trivial we have an isomorphism of abelian
groups
\[
\sN(X) \cong \bigoplus_{i \geq 1} H^{4i} (X ; \ZZ_{(2)}) \oplus
H^{4i-2} (X;\ZZ_2) \oplus KO (X) \otimes \ZZ[\frac{1}{2}]
\]
\end{cor}

Given $n \in \ZZ$ and $G$ a group there is defined an abelian group
$L^s_n (G)$ \cite[chapter 5,6]{Wall(1999)}. For $n = 2k$ it is the
Witt group of based $(-1)^k$-quadratic forms over the group ring
$\ZZ G$, for $n = 2k+1$ it is a certain group of automorphisms of
based $(-1)^k$-quadratic forms over $\ZZ G$. An alternative
description of \cite{Ranicki(1992)} gives these groups uniformly for
all $n$ as cobordism groups of bounded chain complexes of based $\ZZ
G$-modules with an $n$-dimensional Poincar\'e duality. The precise
definition is not that important for us. We are mainly interested in
the invariants which detect these groups for $G \cong \ZZ_N$.
\begin{thm}  \label{L(1)}
For $G = 1$ we have
\[
L^s_n (1) \cong
\begin{cases}
8 \cdot \ZZ & n \equiv 0 \quad (\mod 4) \; (\mathrm{signature}) \\
0 & n \equiv 1 \quad (\mod 4) \\
\ZZ_2 & n \equiv 2 \quad (\mod 4) \; (\mathrm{Arf}) \\
0 & n \equiv 3 \quad (\mod 4)
\end{cases}
\]
\end{thm}
Here `signature' in the last column means that $L^s_{4k} (1) \cong 8
\cdot \ZZ$ is given by the signature of a quadratic form over $\ZZ$,
and `Arf' means that $L^s_{4k+2} (1) \cong \ZZ_2$ is given by the
Arf invariant of a quadratic form over $\ZZ_2$. For $G \neq 1$
functoriality gives maps $L^s_n (1) \ra L^s_n (G)$ and $L^s_n (G)
\ra L^s_n(1)$ yielding the splitting
\begin{equation} \label{reduced-L(G)}
L^s_n (G) \cong L^s_n (1) \oplus \wL^s_n (G).
\end{equation}

Further information about the $L$-groups of finite groups is
obtained using representation theory. For a finite group $G$ complex
conjugation induces an involution on the complex representation ring
$R_{\CC} (G)$. One can define $(\pm 1)$-eigenspaces denoted
$R_{\CC}^{(\pm 1)} (G)$. In terms of characters the
$(+1)$-eigenspace corresponds to real characters, the
$(-1)$-eigenspace corresponds to purely imaginary characters.

A non-degenerate $(-1)^k$-quadratic form over $\ZZ G$ can be
complexified. One can take its associated non-degenerate
$(-1)^k$-symmetric bilinear form and consider the positive and
negative definite $\CC$-vector subspaces. These become
$G$-representations and hence can be subtracted in $R_{\CC} (G)$.
This process  defines the $G$-signature homomorphism (see
\cite[chapter 13]{Wall(1999)} or \cite[chapter 22]{Ranicki(1992)})
\[
\Gsign \co L^s_{2k} (G) \ra R_{\CC}^{(-1)^k} (G).
\]
Its image is $4 \cdot R_{\CC}^{(-1)^k} (G)$. In case $G = \ZZ_N$ for
$N = 2^K$ the $L$-groups are completely calculated (see
\cite{Hambleton-Taylor(2000)}):\footnote{The choice of the notation
in the last line is explained later in section
\ref{sec:rho-invariant}.}

\begin{thm} \label{L(G)}
For $G = \ZZ_N$ we have that
\begin{align*}
L^s_n (G) & \cong
\begin{cases}
4 \cdot R_{\CC}^{(+1)} (G) & n \equiv 0 \; (\mod 4) \; (\Gsign, \;
\mathrm{purely} \; \mathrm{real}) \\
0 & n \equiv 1 \; (\mod 4) \\
4 \cdot R_{\CC}^{(-1)} (G) \oplus \ZZ_2 & n \equiv 2 \; (\mod 4) \;
(\Gsign, \; \mathrm{purely}
\; \mathrm{imaginary}, \mathrm{Arf}) \\
\ZZ_2 & n \equiv 3 \; (\mod 4) \; (\mathrm{codimension} \; 1 \;
\mathrm{Arf})
\end{cases} \\
\widetilde L^s_{2k} (G) & \cong 4 \cdot \RhG^{(-1)^k} \;
\textit{where} \; \RhG^{(-1)^k} \; \textit{is} \; R_{\CC}^{(-1)^k}
(G) \; \textit{modulo the regular representation.}
\end{align*}
\end{thm}

Next we describe briefly the maps in (\ref{ses}). If $n=2k$ the map
$\theta$ is given by first making the degree one normal map $(f,b)
\co M \ra X$ $k$-connected and then taking the quadratic refinement
of the $(-1)^k$-symmetric bilinear form over $\ZZ[G]$ on the kernel
of $f_\ast \co H_k (\widetilde M) \ra H_k (\widetilde X)$. The
exactness at $\sN (X)$ means that there is a degree one normal map
$(f',b') \co M' \ra X$ with $f'$ a homotopy equivalence in the
normal cobordism class of $(f,b)$ if and only if $\theta (f,b) = 0$.

The map $\eta$ is given by taking the stable normal bundle $\nu_M$
of $f \co M \xra{\simeq_s} X$ and associating to it $(f,b) \co M \ra
X$ with $b \co \nu_M \ra (f^{-1})^\ast \nu_M$ induced by $f$.

To describe $\partial$ we need the realization theorem for elements
of $L^s_n (G)$. It says that if $M^{n-1}$ is a manifold and $x \in
L^s_n (G)$ there exists a degree one normal map $(F,B) \co
(W,\partial_0 W, \partial_1 W) \ra (M \times I, M \times 0, M \times
1)$, where $I = [0,1]$, such that $\partial_0 F \co
\partial_0 W \ra M \times 0$ is a homeomorphism, $\partial_1 F
\co \partial_1 W \ra M \times 1$ is a simple homotopy equivalence
and $\theta (F,B) = x$. The `map' $\partial$ in fact means that
there is an action of $L^s_n (G)$ on $\sS^s (X)$ given as follows.
Let $f \co M \ra X \in \sS^s (X)$ and $x \in L^s_n (G)$, then
$\partial (x,f)$ is given by $\partial_1 F_1 \circ f \co \partial_1
W \ra X$ where $(F,B) \co W \ra M \times I$ realizes $x$. When the
abelian group structure of \cite[chapter 18]{Ranicki(1992)} is
imposed on $\sS^s (X)$ the action $\partial$ corresponds to the
group action of the subgroup generated by the image of $\partial$ on
$\sS^s (X)$.

Hence the problem of determining $\sS^s(X)$ in general consists of
determining firstly $\sN (X)$, which is tractable via standard
algebraic topology, secondly $L^s_n (G)$ which we know in our case,
thirdly determining the maps $\partial$, $\eta$, $\theta$ and
finally solving an extension problem which is left over.

\begin{rem} \label{s-vs-h}
One can also define the {\it structure set} $\sS^h (X)$ of an
$n$-manifold $X$. Here, in comparison with the  definition of $\sS^s
(X)$, one replaces simple homotopy equivalences by homotopy
equivalences and the homeomorphism relation by the $h$-cobordism
relation. There is a version of the sequence (\ref{ses}) in this
situation and again the theory of \cite[chapter 18]{Ranicki(1992)}
makes it into a long exact sequence of abelian groups. The obvious
map $\sS^s (X) \ra \sS^h (X)$ is a homomorphism.
\end{rem}

\subsection{Complex projective spaces} \label{subsec:cp}

We also need the discussion of the classification problem for the
complex projective spaces. This is useful also since the discussion
is simpler in this case and will give us a simple example of the
strategy we will need later.

The complex projective space $\CC P^{d-1}$ is defined as the
quotient of the diagonal $S^1$-action on $S^{2d-1} = S^1 \ast \cdots
\ast S^1$ ($d$-factors). As a real manifold it has dimension $2d-2$
and $\pi_1 (\CC P^n) = 1$. Hence from (\ref{L(1)}) we have that the
surgery exact sequence for $\CC P^{d-1}$ becomes the short exact
sequence
\begin{equation} \label{ses-cp^d-1}
0 \ra \sS^s (\CC P^{d-1}) \ra \sN (\CC P^{d-1}) \xra{\theta}
L^s_{2d-2}(1) \ra 0.
\end{equation}
For the normal invariants we have
\begin{equation} \label{ni-cp^d-1}
\sN (\CC P^{d-1}) \cong \bigoplus_{i=1}^{\lfloor(d-1)/2\rfloor}
H^{4i} (\CC P^{d-1};\ZZ) \oplus \bigoplus_{i=1}^{\lfloor d/2
\rfloor} H^{4i-2} (\CC P^{d-1};\ZZ_2).
\end{equation}
Further we can identify the factors
\begin{align}
\bs_{4i} & \co \sN(\CC P^{d-1}) \ra H^{4i} (\CC P^{d-1};\ZZ) \cong \ZZ \cong L_{4i} (1) \\
\bs_{4i-2} & \co \sN(\CC P^{d-1}) \ra H^{4i-2} (\CC P^{d-1};\ZZ_2)
\cong \ZZ_2 \cong L_{4i-2} (1)
\end{align}
as surgery obstructions of degree one normal maps obtained from
$(f,b) \co M \ra \CC P^{d-1}$ by first making $f$ transverse to $\CC
P^{k-1}$ (for $s_{2i}$ where $i = k-1$) and then taking the surgery
obstruction of the degree one map obtained by restricting to the
preimage of $\CC P^i$. The maps $\bs_{2i}$ are called the {\it
splitting invariants}. We will sometimes use (\ref{ni-cp^d-1}) to
identify the elements of $\sN(\CC P^{d-1})$ by $s = (s_{2i})_i$.

The surgery obstruction map $\theta$ takes the top summand of $\sN
(\CC P^{d-1})$ isomorphically onto $L^s_{2d-2} (1)$. Hence the short
exact sequence (\ref{ses-cp^d-1}) splits and we obtain the bijection
of $\sS^s (\CC P^{d-1})$ given by the splitting invariants
$\bs_{2i}$ for $0 < i < d-1$:
\begin{equation} \label{ss-cp^d-1}
\bigoplus _{0 < i < d-1} \bs_{2i} \co \sS^s (\CC P^{d-1})
\xra{\cong} \bigoplus_{0 < i < d-1} L^s_{2i} (1).
\end{equation}
If we think of $\sS^s (\CC P^{d-1})$ as of an abelian group via
Ranicki's identification \cite[chapter 18]{Ranicki(1992)}, then the
map (\ref{ss-cp^d-1}) is an isomorphism.

\subsection{Preliminaries for lens spaces}

When $X$ is a fake lens space $L^{2d-1}(\alpha)$ with $\pi_1
(L^{2d-1}(\alpha)) \cong G = \ZZ_N$ for $N=2^K$ we obtain some
information about the surgery exact sequence for $L^{2d-1}(\alpha)$
from Corollary \ref{ni} and Theorem \ref{L(G)}. In more detail
\begin{equation} \label{ni-lens-spaces}
\sN (L^{2d-1}(\alpha)) \cong \bigoplus_{i=1}^{\lfloor (d-1)/2
\rfloor } H^{4i} (L^{2d-1}(\alpha);\ZZ) \oplus
\bigoplus_{i=1}^{\lfloor d/2 \rfloor} H^{4i-2}
(L^{2d-1}(\alpha);\ZZ_2)
\end{equation}
We denote the factors
\begin{align}
\bt_{4i} & \co \sN(L^{2d-1}(\alpha)) \ra H^{4i} (L^{2d-1}(\alpha);\ZZ) \cong \ZZ_{2^K} \\
\bt_{4i-2} & \co \sN(L^{2d-1}(\alpha)) \ra H^{4i-2}
(L^{2d-1}(\alpha);\ZZ_2) \cong \ZZ_2
\end{align}
and similarly as above we will sometimes use (\ref{ni-lens-spaces})
to identify the elements of $\sN(L^{2d-1}(\alpha))$ by $t =
(t_{2i})_i$. More information is obtained from the following
\begin{thm}[\cite{Wall(1999)}]
\
\begin{enumerate}
\item If $d=2e$, then the map
\[
\theta \co \sN(L^{2d-1}(\alpha)) \ra L^s_{2d-1}(G) = L^s_{4e-1}(G) =
\ZZ_2
\]
is given by $\theta (x) = \bt_{4e-2} (x) \in  \ZZ_2$.
\item
The map
\[
\theta \co \sN_\partial(L^{2d-1}(\alpha) \times I) \ra L^s_{2d}(G)
\]
maps onto the summand $L^s_{2d}(1)$.
\end{enumerate}
\end{thm}

Hence we obtain the short exact sequence
\begin{equation} \label{ses-lens-2d-1}
0 \ra \wL^s_{2d} (G) \xra{\partial} \sS^s   (L^{2d-1}(\alpha))
\xra{\eta} \widetilde{\sN}(L^{2d-1}(\alpha)) \ra 0
\end{equation}
where
\begin{align*}
\widetilde{\sN}(L^{4e-1}(\alpha)) & = \mathrm{ker} \; \big (
\bt_{4e-2} \co
{\sN}(L^{4e-1}(\alpha)) \ra H^{4e-2} (L^{4e-1}(\alpha);\ZZ_2) \cong \ZZ_2 \big ), \\
\widetilde{\sN}(L^{4e+1}(\alpha)) & = \sN(L^{4e+1}(\alpha)).
\end{align*}
in other words
\begin{equation} \label{red-ni-lens-spaces}
\widetilde \sN (L^{2d-1}(\alpha)) \cong \bigoplus_{i=1}^{c} \ZZ_N
\oplus \bigoplus_{i=1}^{c} \ZZ_2
\end{equation}
where $c = \lfloor (d-1)/2 \rfloor$. The first term in the sequence
(\ref{ses-lens-2d-1}) is understood by Theorem \ref{L(G)}, the third
term is understood by (\ref{red-ni-lens-spaces}). Hence we are left
with an extension problem.

\subsection{The Join Construction} \label{subsec:join}

\

\medskip

We will make use of the join construction from \cite[chapter
14A]{Wall(1999)}. It can be explained as follows. Let $G$ be a group
(in our case $G \leq S^1$) acting freely on the spheres $S^m$ and
$S^n$. Then the two actions extend to the join $S^{m+n+1} \cong S^m
\ast S^n$ and the resulting action remains free. When we are given
two lens spaces (complex projective spaces) $L$ and $L'$, we can
pass to universal covers ($S^1$-bundles), form the join and then
pass to the quotient again. The resulting space is again a fake lens
space (a fake complex projective space). This operation will be
denoted $L \ast L'$ and it will be called the {\it join}. When $L' =
L^1(\alpha_1)$ we call this operation a {\it suspension}.

The join with $L^1(\alpha_k)$ defines a map $\Sigma_k \co \sS^s
(L^{2d-1}(\alpha_1)) \lra \sS^s(L^{2d+1}(\alpha_k))$. The inclusion
$L^{2d-1}(\alpha_1) \subset L^{2d+1}(\alpha_k)$ induces a
restriction map on the normal invariants $\res \co \sN
(L^{2d+1}(\alpha_k)) \lra \sN (L^{2d-1}(\alpha_1))$ and we have a
commutative diagram \cite[Lemma 14A.3]{Wall(1999)}:
\begin{equation} \label{susp-diagram}
\begin{split}
\xymatrix{
\sS^s (L^{2d-1}(\alpha_1)) \ar[r]^{\eta} \ar[d]_{\Sigma_k} & \sN (L^{2d-1}(\alpha_1)) \\
\sS^s (L^{2d+1}(\alpha_k)) \ar[r]^{\eta} & \sN (L^{2d+1}(\alpha_k))
\ar[u]_{\res} }
\end{split}
\end{equation}
Note that we have $t_{2i} = \res (t_{2i})$. Hence the map
\begin{equation}
\res \co \widetilde \sN(L^{2d+1}(\alpha_1)) \lra \widetilde
\sN(L^{2d-1} (\alpha_1))
\end{equation}
is an isomorphism when $d = 2e+1$ and it is onto when $d = 2e$ with
the kernel equal to $\ZZ_N(t_{4e})$. A similar diagram exists for
the situation $\CC P^d = \CC P^{d-1} \ast \mathrm{pt}$.

The map $\Sigma_k$ is a homomorphism when the structure sets are
equipped with the abelian groups structure from \cite[chapter
18]{Ranicki(1992)}. To see this notice that
\[
L^{2d+1}(\alpha_k) = E(\nu) \cup_{S(\nu)} C
\]
where $E(\nu)$ is the total space of the normal disk-bundle of
$L^{2d-1}(\alpha_1)$ in $L^{2d+1}(\alpha_k)$, $S(\nu)$ is the
associated sphere-bundle and $C$ is the complement (it is the total
space of a disk-bundle over $L^1(\alpha_k)$). Then there is a
commutative diagram
\[
\xymatrix{
\sS^s (L^{2d-1}(\alpha_1)) \ar[d]_{\Sigma_k} \ar[dr]^{\nu^{!}} & \\
\sS^s (L^{2d+1}(\alpha_k)) \ar[r]_{\cong} & \sS^s (E(\nu),S(\nu)) }
\]
The map in the bottom row is obtained using \cite[Theorem
12.1]{Wall(1999)}. It follows from the calculation $\sS^s_\partial
(C) = 0$ that it is an isomorphism. The map $\nu^{!}$ is the
transfer map obtained via pullback. This coincides with the
algebraic surgery transfer map from \cite[chapter
21]{Ranicki(1992)}.\footnote{We thank A. Ranicki for informing us
about the last claim.}


\section{The $\rho$-invariant} \label{sec:rho-invariant}


We review the definition of the $\rho$-invariant for odd-dimensional
manifolds and some of its properties from
\cite{Atiyah-Singer-III(1968)} and \cite{Wall(1999)}. It will
provide us with a map from the short exact sequence
(\ref{ses-lens-2d-1}) to a certain short exact sequence coming from
representation theory of $G$. Studying this map will enable us to
solve the extension problem we are left with in the next section.

\subsection{Definitions}

\

\medskip

Let $G$ be a compact Lie group acting smoothly on a smooth manifold
$Y^{2d}$. The middle intersection form becomes a non-degenerate
$(-1)^d$-symmetric bilinear form on which $G$ acts. As explained
earlier, such a form yields an element in the representation ring
$R(G)$ which we denote by $\Gsign (Y)$. The discussion in section
\ref{sec:ses} also tells us that we have $\Gsign (Y) \in R^{(-1)^d}
(G)$ which in terms of characters means that we obtain a real
(purely imaginary) character, which will be denoted as $\Gsign (-,Y)
\co g \in G \mapsto \Gsign (g,Y) \in \CC$. The (cohomological
version of the) Atiyah-Singer $G$-index theorem \cite[Theorem
(6.12)]{Atiyah-Singer-III(1968)} tells us that if $Y$ is closed then
for all $g \in G$
\begin{equation} \label{ASGIT}
\Gsign (g,Y)  = L(g,Y) \in \CC,
\end{equation}
where $L(g,Y)$ is an expression obtained by evaluating certain
cohomological classes on the fundamental classes of the $g$-fixed
point submanifolds $Y^g$ of $Y$. In particular if the action is free
then $\Gsign (g,Y) = 0$ if $g \neq 1$. This means that $\Gsign (Y)$
is a multiple of the regular representation. This theorem was
generalized by Wall to topological semifree actions on topological
manifolds, which is the case we will need in this paper
\cite[chapter 14B]{Wall(1999)}. The assumption that $Y$ is closed is
essential here, and motivates the definition of the
$\rho$-invariant. In fact, Atiyah and Singer provide two
definitions. For the first one one also needs the result of Conner
and Floyd \cite{Conner-Floyd(1964)} that for an odd-dimensional
manifold $X$ with a finite fundamental group there always exists a
$k \in \NN$ and a manifold with boundary $(Y,\partial Y)$ such that
$\pi_1 (Y) \cong \pi_1 (X)$ and $\partial Y = k \cdot X$.

\begin{defn}{\cite[Remark after Corollary 7.5]{Atiyah-Singer-III(1968)}} \label{defn-rho-1}
Let $X^{2d-1}$ be a closed manifold with $\pi_1 (X) \cong G$ a
finite group. Define
\begin{equation}
\rho (X) = \frac{1}{k} \cdot \Gsign(\widetilde Y) \in \QQ R^{(-1)^d}
(G)/ \langle \textup{reg} \rangle
\end{equation}
for some $k \in \NN$ and $(Y,\partial Y)$ such that $\pi_1 (Y) \cong
\pi_1 (X)$ and $\partial Y = k \cdot X$. The symbol $\langle
\textup{reg} \rangle$ denotes the ideal generated by the regular
representation.
\end{defn}
\noindent By the Atiyah-Singer $G$-index theorem \cite[Theorem
(6.12)]{Atiyah-Singer-III(1968)} is $\rho$ well defined.
\begin{defn} \label{defn-rho-2}
Let $G$ be a compact Lie group acting freely on a manifold
$\wX^{2d-1}$. Suppose in addition that there is a manifold with
boundary $(Y,\partial Y)$ on which $G$ acts (not necessarily freely)
and such that $\partial Y = \wX$. Define
\[
\rho_G (\wX) \co g \in G \mapsto \Gsign (g,Y) - L(g,Y) \in \CC.
\]
\end{defn}
In this definition we think about the $\rho$-invariant as about a
function $G \smin \{1\} \ra \CC$. When both definitions apply (that
means when $G$ is a finite group), then they coincide, that means
$\rho (X) = \rho_G (\wX)$.

For finite $G < S^1$ we will use special notation following
\cite[Proof of Proposition 14E.6 on page 222]{Wall(1999)}. By $\Gh$
is denoted the Pontrjagin dual of $G$, the group $\Hom_\ZZ (G,S^1)$.
Recall that for a finite cyclic $G$ the representation ring $R(G)$
can be canonically identified with the group ring $\ZZ \Gh$. Then we
also have $\QQ R(G) = \QQ \otimes R(G) = \QQ \Gh$. Dividing out the
regular representation  corresponds to dividing out the norm
element, denoted by $Z$, hence $R(G)/\langle \textup{reg} \rangle =
R_{\Gh} = \ZZ\Gh/\langle Z \rangle$ and $\QQ R(G)/\langle
\textup{reg} \rangle = \QQ R_{\Gh} = \QQ \Gh/\langle Z \rangle$.
Choosing a generator $\Gh = \langle \chi \rangle$ gives the
identifications $\QQ \RhG = \QQ [\chi]/\langle 1+\chi +\cdots +
\chi^{N-1} \rangle$ where $N$ is the order of $G$. In order to save
space we also use the following notation $I \langle K \rangle =
\langle 1 + \chi + \cdots + \chi^{N-1} \rangle$.

Recall that $R(G)$ contains two eigenspaces $R(G)^{\pm}$ with respect to the conjugation action. In terms of the above identification of $R(G)$ and $\RhG$ with the polynomial rings we have identifications:
\begin{align*}
 \RhG^- & = \langle \chi^k - \chi^{N-k} \; | \; k = 1, \ldots, (N/2)-1 \rangle \\
 &= \{ \; p \in \ZZ[\chi]/ I \langle K \rangle \; | \; p (\chi^{N-1}) = - p (\chi) \;  \},  \\
 \RhG^+ & = \langle \chi^k + \chi^{N-k} \; | \; k = 0, \ldots, (N/2)-1 \rangle \\
 &= \{ \; p \in \ZZ[\chi]/ I \langle K \rangle \; | \; p (\chi^{N-1}) = p (\chi) \; \textup{and} \;  p(-1) \equiv 0 \; \mod 2 \; \}.  
\end{align*}

\subsection{Properties}

\

\medskip

The $\rho$-invariant is an $h$-cobordism invariant \cite[Corollary
7.5]{Atiyah-Singer-III(1968)}. For $X^{2d-1}$ with $\pi_1 (X) \cong
G$ it defines a function of $\sS^s (X)$ by sending $a = [h \co M
\lra X]$ to $\wrho (a) = \rho(M) - \rho (X)$. If we put on $\sS^s
(X)$ the abelian group structure from \cite[chapter
18]{Ranicki(1992)} it is not clear whether $\wrho$ is a homomorphism
in general.\footnote{We will see below that it is a homomorphism when $X = L^{2d-1}(\alpha)$.} Still the following property holds always.

\begin{prop}
For $X^{2d-1}$ with $\pi_1 (X) \cong G$ there is a commutative
diagram
\[
\xymatrix{
L^s_{2d} (G) \ar[r]^{\partial} \ar[d]^{\Gsign} & \sS^s(X) \ar[d]^{\wrho} \\
4 \cdot R_{\CC}^{(-1)^d} (G) \ar[r] & \QQ R_{\CC}^{(-1)^d}
(G)/\langle \textup{reg} \rangle . }
\]
Moreover, for $z \in L^s_{2d} (G)$ and $x \in \sS^s(X)$ we have
\[
\wrho (x+\partial z) = \wrho (x) + \wrho (\partial z).
\]
\end{prop}

\begin{proof}
See \cite[Theorem 2.3]{Petrie(1970)}. It essentially follows from
definitions. We also use the identification of the geometrically
given action of $L^s_{2d} (G)$ on $\sS^s(X)$ with the action coming
from the abelian group structure on $\sS^s(X)$ of \cite[chapter
18]{Ranicki(1992)}.
\end{proof}

\begin{rem} \label{rho-factors}
The map $\wrho$ also obviously factors through the map $\sS^s (X)
\ra \sS^h (X)$ of Remark \ref{s-vs-h}.
\end{rem}

When $X = L^{2d-1}(\alpha_k)$, it follows from the above diagram,
the exactness of the surgery exact sequence, the Atiyah-Singer
$G$-index theorem and the calculation of the groups $L^s_{2d} (G)$
that the action of $\wL^s_{2k} (G)$ on $\sS^s(L^{2k-1})$ is free. In
fact we have

\begin{prop}\label{ses-vs-rep-thy}
There is the following commutative diagram of abelian groups and
homomorphisms with exact rows
\[
\xymatrix{ 0 \ar[r] & \wL^s_{2d} (G) \ar[r]^(0.4){\partial}
\ar[d]_{\cong}^{\Gsign} & \sS^s (L^{2d-1}(\alpha)) \ar[r]^{\eta}
\ar[d]^{\widetilde \rho}&
\widetilde \sN (L^{2d-1}(\alpha)) \ar[r] \ar[d]^{[\widetilde \rho]}& 0 \\
0 \ar[r] & 4 \cdot R^{(-1)^d}_{\widehat G} \ar[r] & \QQ
R^{(-1)^d}_{\widehat G} \ar[r] & \QQ R^{(-1)^d}_{\widehat G}/ 4
\cdot R^{(-1)^d}_{\widehat G} \ar[r] & 0 }
\]
where $[\wrho]$ is the homomorphism induced by $\wrho$.
\end{prop}

All the statements follow from the previous discussion except the
claim that $\wrho$ and  $[\wrho]$ are homomorphisms. This will be
proved in this section, first for $\alpha_1$, then for $\alpha_k$,
and finally for general $\alpha$. To this end we need some way to
calculate the $\rho$-invariant for fake lens spaces. The formulas we
obtain will give us first a good understanding of the map $[\wrho]$.
Using certain naturality properties we will obtain also the claim
about $\wrho$.

Recall the join $L \ast L'$ of the lens spaces $L$ and $L'$ from
section \ref{subsec:join}. We have \cite[chapter 14A]{Wall(1999)}
\begin{equation}
\rho (L \ast L') = \rho (L) \cdot \rho (L').
\end{equation}
For $L^1(\alpha_k)$ we have \cite[Proof of Theorem
14C.4]{Wall(1999)}
\begin{equation}
\rho (L^1(\alpha_k)) = f_k \in \QQ \RhGm
\end{equation}
where $f_k$ is defined as follows.

\begin{defn}
For odd $k \in \NN$ we set
\[
f_k := \frac{1+\chi^k}{1-\chi^k} \mathrm{\qquad and \qquad} f'_k :=
\frac{1-\chi+\chi^2-\cdots-\chi^{k-2}+\chi^{k-1}}{1+\chi+\chi^2+\cdots+\chi^{k-2}+\chi^{k-1}}.
\]
We abbreviate $f := f_1$.
\end{defn}

\begin{lem}\label{f_k-lem}
Let $G = \ZZ_N$ with $N=2^K$. For odd $k \in \NN$ we have
\[
f_k \in \QQ \RhGm, \qquad f_k = f \cdot f'_k, \qquad f'_k \in \RhG.
\]
\end{lem}
\begin{proof}
Notice that $1-\chi^k$ is invertible in $\QQ \RhG$ because
\[
(1-\chi^k)^{-1} = -\frac{1}{N}(1 + 2 \cdot \chi^k + 3 \cdot
\chi^{2k} + \cdots + N \cdot \chi^{(N-1)k}) \in \QQ \RhG.
\]
Therefore $f_k \in \QQ \RhG$ and the identity
\[
\frac{1+\chi^{-k}}{1-\chi^{-k}} = - \frac{1+\chi^k}{1-\chi^k} = -
f_k
\]
implies $f_k \in \QQ \RhGm$. An easy calculation shows $f_k = f
\cdot f'_k$. That $f'_k \in \RhG$ follows from the fact that
$1+\chi+\chi^2+\cdots+\chi^{k-1}$ is invertible in $\RhG$. The
inverse is given by $1+\chi^k+\chi^{2k}+\cdots+\chi^{(r-1)k}$ where
$r$ denotes a natural number such that $r \cdot k-1$ is a multiple of
$N=2^K$.
\end{proof}
Also a formula of Wall which calculates the $\rho$-invariant for
fake complex projective spaces will be useful. Let $a = [h \co Q
\lra \CC P^{d-1}]$ be an element of $\sS^s (\CC P^{d-1})$ and let
$\widetilde h \co \widetilde Q \lra S^{2d-1}$ be the associated map
of $S^1$-manifolds. Denote $\wrho_{S^1} (a) := \wrho_{S^1}
(\widetilde Q) - \wrho_{S^1} (S^{2d-1})$ defining a function of
$\sS^s (\CC P^{d-1})$.

\begin{thm}{\cite[Theorem 14C.4]{Wall(1999)}} \label{rho-formula-cp}
Let $a = [h \co Q \ra \CC P^{d-1}]$ be an element in $\sS^s (\CC
P^{d-1})$. Then for $t \in S^1$
\[
\wrho_{S^1} (t,a) = \sum_{1 \leq i \leq \lfloor d/2 \rfloor -1} 8
\cdot \bs_{4i} (\eta (a)) \cdot (f^{d-2i} - f^{d-2i-2}) \in \CC,
\]
where $f = (1+t)/(1-t)$.
\end{thm}

Among other things this also shows that $\wrho_{S^1}$ is a
homomorphism of $\sS^s (\CC P^{d-1})$. Coming back to lens spaces
recall that there is an $S^1$-bundle (better $L^1(\alpha_1)$-bundle)
$p \co L^{2d-1}(\alpha_1) \lra \CC P^{d-1}$. Via pullback it induces
a commutative diagram
\begin{equation} \label{res-diag}
\begin{split}
\xymatrix{
0 \ar[r] & \sS^s (\CC P^{d-1} ) \ar[d]^{p^{!}} \ar[r]^{\eta} & \sN (\CC P^{d-1}) \ar[d]^{p^{!}} \ar[r] & L_{2(d-1)} (1) \\
\wL^s_{2d} (G) \ar[r] & \sS^s (L^{2d-1}(\alpha_1)) \ar[r]^{\eta} &
\widetilde \sN (L^{2d-1}(\alpha_1)) \ar[r] & 0 }
\end{split}
\end{equation}
With the abelian group structure of \cite[chapter 18]{Ranicki(1992)}
the maps $p^{!}$ are homomorphisms by the identification of
geometric and algebraic transfers. Another way of thinking about
$p^{!}$ is that it is given by passing to the subgroup $G < S^1$.
Since the $\rho$-invariant is natural for passing to subgroups we
obtain
\begin{cor}{\cite[Theorem 14E.8]{Wall(1999)}} \label{rho-formula-1-lens-sp}
Let $a \in \sS^s (L^{2d-1}(\alpha_1))$ such that $a = p^{!} (b)$ for
some $b \in \sS^s (\CC P^{d-1})$. Then
\[
\wrho (a) = \sum_{1 \leq i \leq \lfloor d/2 \rfloor -1} 8 \cdot
\bs_{4i} (\eta (b)) \cdot (f^{d-2i} - f^{d-2i-2}) \in \QQ
\RhG^{(-1)^d},
\]
where $f = (1+\chi)/(1-\chi)$.
\end{cor}

For the map $p^{!} \co \sN (\CC P^{d-1}) \lra \widetilde \sN
(L^{2d-1} (\alpha_1))$ we have
\begin{equation}
p^{!} (s_{4i-2}) = t_{4i-2} \quad p^{!} (s_{4i}) = t_{4i}
\end{equation}
and hence it is surjective. If the map $\sS^s (\CC P^{d-1}) \ra \sN
(\CC P^{d-1}) \ra \sN (L^{2d-1}(\alpha_1))$ were surjective we could
use Corollary \ref{rho-formula-1-lens-sp} to give a formula for the
function $[\wrho]$. This is the case when $d = 2e$. In the case $d =
2e+1$ all the summands but the $\ZZ_N ( t_{4e})$ from $\sN
(L^{2d-1}(\alpha_1))$ are hit from $\sS (\CC P^{d-1})$. We need the
following

\begin{lem} \label{last-summand}
Let $d = 2e+1$ and let $a \in \sS (L^{2d-1}(\alpha_1))$ be such that
$a \mapsto \bt(\eta(a))=(0,\ldots,1) \in \sN (L^{2d-1}(\alpha_1))$,
i.e. $\bt(\eta(a))_{4i} = 0$ for $i \leq e-1$ and
$\bt(\eta(a))_{4e}=1$. Then
\[
\widetilde\rho (a) = 8 f + z \quad \in \quad \QQ\RhGm
\]
for some $z \in 4 \cdot \RhGm$.
\end{lem}

\begin{proof}
We will use the suspension map $\Sigma_1$ from section
\ref{subsec:join}. Our assumptions mean that $\bt (\eta(a))$ is not
in the image of the composition $\sS (\CC P^{d-1}) \ra \sN (\CC
P^{d-1}) \ra \sN (L^{2d-1}(\alpha_1))$. However, diagram
(\ref{susp-diagram}) tells us that $\bt (\eta (\Sigma_1(a)))$ is in
the image of $\sS (\CC P^{(d+1)-1}) \ra \sN (\CC P^{(d+1)-1}) \ra
\sN (L^{2(d+1)-1}(\alpha_1))$ and hence we have
\[
f \cdot \widetilde\rho (a) + y = 8 \cdot 1 \cdot (f^2 -1) \quad \in
\quad \QQ\RhGp
\]
for some $y \in 4 \cdot \RhGp$. We obtain the desired identity by
the following calculation. Let $\hrho \in \QQ[\chi]$ and $\hy \in 4
\cdot \ZZ[\chi]$ be representatives for $\wrho (a)$ and $y$. Then
\begin{align*}
(1+\chi)(1-\chi)\hrho + (1-\chi)^2 \hy & \equiv 8 \cdot (4 \chi) \quad \quad \quad \mod I \langle K \rangle \\
(1+\chi)(1-\chi)\hrho + (1-\chi)^2 (\hy+8) & \equiv 8 \cdot (1+\chi)^2 \quad \; \mod I \langle K \rangle \\
(1+\chi)(1-\chi)\hrho + (1-\chi)^2 (\hy+8) & = 8 \cdot (1+\chi)^2 +
g(\chi) (1 + \chi + \cdots + \chi^{N-1}) \in \QQ[\chi]
\end{align*}
for some $g(\chi) \in \QQ[\chi]$. Hence $\hy+8 = (1+\chi) w(\chi)$
for some $w(\chi) \in \QQ[\chi]$. Since $(\hy + 8) \in 4 \cdot \ZZ
[\chi]$, we obtain $w(\chi) \in 4 \cdot \ZZ[\chi]$. Further write $g
(\chi) = 2r+(1+\chi) g'(\chi) = r (1-\chi) + (1+\chi) (r+g'(\chi))$
for $r \in \QQ$, $g' (\chi) \in \QQ[\chi]$. We have
\begin{align*}
(1-\chi)\hrho + (1-\chi)^2 w(\chi) & = 8 \cdot (1+\chi) + g(\chi) (1
+ \chi^2 + \cdots \chi^{N-2}) \in \QQ[\chi]
\end{align*}
and further modulo $I \langle K \rangle$
\begin{align*}
(1-\chi)\hrho + (1-\chi)^2 w(\chi) & \equiv 8 \cdot (1+\chi) + r (1-\chi) (1 + \chi^2 + \cdots \chi^{N-2})  \\
\hrho + (1-\chi) w(\chi) & \equiv 8 \cdot f + r (1 + \chi^2 + \cdots
\chi^{N-2})
\end{align*}
Now $(1-\chi)w(\chi) = (2-(1+\chi))w(\chi) = 2w(\chi) - (\hy + 8)$.
Further $2w(\chi) = w^+(\chi) + w^-(\chi)$, where $w^\pm(\chi) := w(\chi) \pm w(\chi^{-1}) \in 4 \cdot \ZZ [\chi] / I \langle K \rangle $. Hence
\begin{equation*}
\wrho (a) - 8 \cdot f + w^-(\chi) = (\hy+8) - w^+ (\chi) + r (1 +
\chi^2 + \cdots + \chi^{N-2})
\end{equation*}
in $\QQ[\chi] / I \langle K \rangle$, while the left hand side of
the equation lies in the $(-1)$-eigenspace and the right-hand side
lies in the $(+1)$-eigenspace and hence both are equal to $0$. It
follows that
\[
\wrho (a)= 8 \cdot f - w^-(\chi).
\]
Putting $z = - w^-(\chi)$ yields the desired formula.
\end{proof}

\begin{lem} \label{all-summands-d=2e+1}
Let $d = 2e+1$ and let $a \in \sS (L^{2d-1}(\alpha_1))$. Then
\[
\widetilde\rho (a) = 8 \cdot \bt_{4e} (\eta(a)) \cdot f + \!\!
\sum_{1 \leq i \leq \lfloor d/2 \rfloor -1} \!\! 8 \cdot \bt_{4i}
(\eta(a)) \cdot (f^{d-2i} - f^{d-2i-2}) + z \quad \in \quad \QQ\RhGm
\]
for some $z \in 4 \cdot \RhGm$.
\end{lem}

\begin{proof}
Proof is by a straightforward modification of the proof of Lemma
\ref{last-summand}.
\end{proof}

\begin{prop} \label{rho-formula-lens-sp}
For the map $[\wrho] \co \widetilde \sN (L^{2d-1}(\alpha_1)) \lra
\QQ \RhG^{(-1)^d} / 4 \cdot \RhG^{(-1)^d}$ and an element $t =
(t_{2i})_i \in \widetilde \sN (L^{2d-1}(\alpha_1))$ we have that
\begin{align*}
d = 2e \; : \; [\widetilde \rho] (t) & = \sum_{i=1}^{e-1} 8 \cdot
t_{4i} \cdot f^{d-2i-2} \cdot (f^2-1) \\
d = 2e+1 \; : \; [\widetilde \rho] (t) & = \sum_{i=1}^{e-1} 8 \cdot
t_{4i} \cdot f^{d-2i-2} \cdot (f^2-1) + 8 \cdot t_{4e} \cdot f.
\end{align*}
\end{prop}

\begin{proof}
It is enough to find for each $t \in \widetilde \sN
(L^{2d-1}(\alpha_1))$ some $a \in  \sS^s (L^{2d-1}(\alpha_1))$ with
$\bt(\eta(a))=t$ and for which we can calculate $\wrho (a) \in \QQ
\RhG^{(-1)^d}$. If $d = 2e$ then by discussion after Corollary
\ref{rho-formula-1-lens-sp} there is for each normal cobordism class
a fake lens space which fibers over a fake complex projective space
and hence the formula from Corollary \ref{rho-formula-1-lens-sp}
gives the desired formula. If $d = 2e+1$ then the same reasoning
applied to Lemma \ref{all-summands-d=2e+1} gives the desired
formula.
\end{proof}

\begin{cor} \label{rho-homomorphism}
The function $\wrho \co \sS^s(L^{2d-1}(\alpha_1)) \lra \QQ
\RhG^{(-1)^d}$ is a homomorphism.
\end{cor}

\begin{proof}
It is enough to show that for every $t$, $t' \in \widetilde \sN
(L^{2d-1}(\alpha_1))$ there exist elements (not necessarily unique)
$a$, $a'$ in $\sS^s (L^{2d-1}(\alpha_1))$ such that $\bt (\eta(a)) =
t$, $\bt (\eta(a')) = t'$ and $\wrho (a+a') = \wrho (a) + \wrho
(a')$. If this holds, then for any $x,x' \in \sS^s
(L^{2d-1}(\alpha_1))$ choose $a$ and $a'$ as above corresponding to
the classes $\bt (\eta(x))$, $\bt (\eta(y)) \in \sN (L^{2d-1}
(\alpha_1))$. Then $x = a +\partial(b)$ and $x' = a' +
\partial(b')$ for some $b$, $b' \in \partial \wL^s_{2d} (G)$ and
\begin{align*}
\wrho (x+x') = & \wrho (a + \partial b + a' + \partial b') = \wrho
(a + a') +
\wrho (\partial b + \partial b') \\
= &  \wrho (a) + \wrho(a') + \wrho (\partial b) + \wrho(\partial b')
= \wrho (x) + \wrho(x').
\end{align*}

When $d=2e$ we can associate to a given $t \in \sN
(L^{2d-1}(\alpha_1))$ an $a \in \sS^s (L^{2d-1}(\alpha_1))$ coming
from the $\sS (\CC P^{d-1})$, i.e. $a = p^{!} (b)$ where $b \in \sS
(\CC P^{d-1})$ such that $p^{!} (\eta (b)) = t$. When we have $t$,
$t'\in \sN (L^{2d-1}(\alpha_1))$, then $\wrho (a+a') = \wrho (p^{!}
(b) + p^{!} (b')) = \wrho (p^{!} (b + b')) = \res (\wrho_{S^1} (b +
b')) = \res (\wrho_{S^1} (b) + \wrho_{S^1} (b')) = \res (\wrho_{S^1}
(b)) + \res(\wrho_{S^1} (b')) = \wrho (a) + \wrho (a')$. Here $\res$
denotes the map on the representation rings induced by the inclusion
$G < S^1$.

When $d = 2e+1$ and $t \in \sN (L^{2d-1}(\alpha_1))$ we can do the
same unless $t_{4e} \neq 0$. In that case there is no fake lens
space in the normal cobordism class of $\bt$ which fibers over a
fake complex projective space and we have to use a different
argument. It follows from the formula in Proposition
\ref{rho-formula-lens-sp} that for $a, a' \in \sS^s
(L^{2d-1}(\alpha_1))$ we have $\wrho (a+a') = \wrho (a) + \wrho (a')
+ z$ for some $z \in 4 \cdot \RhG^-$. If $a$, $a'$ are in the same
normal cobordism class then $z = 0$. Our task is to show this for
any choice of $a$, $a'$. We use the fact that $\Sigma$ is a
homomorphism and that we have already proved the claim for $d =
2e+2$. That implies $\wrho (\Sigma (a+a')) = \wrho (\Sigma a +
\Sigma a')) = \wrho (\Sigma a) + \wrho (\Sigma a') = f \cdot \wrho
(a) + f \cdot \wrho(a')$. On the other hand $\wrho (\Sigma (a+a')) =
f \cdot \wrho (a+a') = f \cdot \wrho (a) + f \cdot \wrho(a') + f
\cdot z$. Hence it is enough to show that for any $z \in 4 \cdot
\RhG^-$ such that $f \cdot z = 0$ in $\QQ \RhG^+$ we have $z = 0$ in
$4 \cdot \RhG^-$. This is proved below in Lemma
\ref{injectivity-mult-f}.
\end{proof}

Now we proceed to the case of $\alpha_k$ where $k \in \NN$ is odd.

\begin{prop} \label{rho-formula-lens-sp-k}
For the map $[\wrho] \co \widetilde \sN (L^{2d-1}(\alpha_k)) \lra
\QQ \RhG^{(-1)^d} / 4 \cdot \RhG^{(-1)^d}$ an element $t =
(t_{2i})_i \in \widetilde \sN (L^{2d-1}(\alpha_k))$ we have that
\begin{align*}
d = 2e \; : \; [\widetilde \rho] (t) & = \sum_{i=1}^{e-1} 8 \cdot
t_{4i} \cdot f'_k \cdot f^{d-2i-2} \cdot (f^2-1) \\
d = 2e+1 \; : \; [\widetilde \rho] (t) & = \sum_{i=1}^{e-1} 8 \cdot
t_{4i} \cdot f'_k \cdot f^{d-2i-2} \cdot (f^2-1) + 8 \cdot t_{4e}
\cdot f'_k \cdot f.
\end{align*}
\end{prop}

\begin{proof}
We will use the calculation for $\alpha_1$ and the homeomorphisms
\[
L^{2d+1}(\alpha_k) \cong L^{2d-1}(\alpha_1) \ast L^1(\alpha_k) \quad
\mathrm{and} \quad L^{2d+1}(\alpha_k) \cong L^{2d-1}(\alpha_k) \ast
L^1(\alpha_1)
\]
For $d=2e$ recall the diagram
\begin{equation}
\begin{split}
\xymatrix{ \QQ \RhG^- \ar[d]^{\cdot f_k} & \sS^s (L^{4e-3}
(\alpha_1)) \ar[l]_(0.6){\wrho}
\ar[d]^{\Sigma_k} \ar@{-{>>}}[r]^{\eta} & \widetilde \sN (L^{4e-3} (\alpha_1)) \\
\QQ \RhG^+ & \sS^s (L^{4e-1} (\alpha_k)) \ar[l]_(0.6){\wrho}
\ar@{-{>>}}[r]^{\eta} & \widetilde  \sN (L^{4e-1} (\alpha_k))
\ar[u]_{\res}^{\cong} }
\end{split}
\end{equation}
Let $t \in \widetilde  \sN (L^{4e-1}(\alpha_k))$. Choose $x \in
\sS^s(L^{4e-3}(\alpha_1))$ such that $\bt (\eta (x)) = t = \res
(t)$. Then we have $\bt (\eta (\Sigma_k x)) = t$ and $[\wrho] (\eta
(\Sigma_k x)) = [\wrho (x) \cdot f_k]$ can be calculated using the
formulas from the case $k=1$.

For $d=2e+1$ recall the diagram
\begin{equation}
\begin{split}
\xymatrix{ \QQ \RhG^- \ar[d]^{\cdot f} & \sS^s (L^{4e+1} (\alpha_k))
\ar[l]_(0.6){\wrho} \ar[d]^{\Sigma_1} \ar@{-{>>}}[r]^{\eta} & \widetilde  \sN (L^{4e+1} (\alpha_k)) \\
\QQ \RhG^+ & \sS^s (L^{4e+3} (\alpha_k)) \ar[l]_(0.6){\wrho}
\ar@{-{>>}}[r]^{\eta} & \widetilde  \sN (L^{4e+3} (\alpha_k))
\ar[u]_{\res}^{\cong} }
\end{split}
\end{equation}
Let $t \in \widetilde  \sN (L^{4e+1}(\alpha_k))$. Choose $x \in
\sS^s(L^{4e+1}(\alpha_1))$ such that $\bt (\eta (x)) = t$. Then we
have $\bt (\eta (\Sigma_1 x)) = t$ and $\wrho (\Sigma_1 x) = \wrho
(x) \cdot f$. We obtain the equation
\[
f \cdot \wrho(x) + y = \sum_{i=1}^{e-1} 8 \cdot t_{4i} \cdot f'_k
\cdot f^{d+1-2i-2} \cdot (f^2-1) + 8 \cdot t_{4e} \cdot f'_k \cdot
(f^2 -1) \in  \QQ \RhG^+
\]
for some $y \in 4 \cdot \RhG^+$ using the formulas from the case $d
= 2e+2$ which we have already dealt with. Now a modification of the
argument from the proof of Lemma \ref{last-summand} can be used to
obtain the formula for $[\wrho] (\eta (x))$.
\end{proof}

\begin{cor}
The function $\wrho \co \sS^s(L^{2d-1}(\alpha_k)) \lra \QQ
\RhG^{(-1)^d}$ is a homomorphism.
\end{cor}

\begin{proof}
Just as in the case $\alpha_1$ we have to show that the addition of
elements in different normal cobordism classes works. For this it is
enough to find suitable representatives. In the case $d=2e$ we can
choose in each normal cobordism class an element coming from $\sS^s
(L^{4e-3} (\alpha_1))$. In the case $d=2e+1$ there is again a
problem with the summand $\ZZ_N (t_{4e})$ which can be resolved by
the same reasoning as in the case $\alpha_1$.
\end{proof}

\begin{cor}
The function $\wrho \co \sS^s(L^{2d-1}(\alpha)) \lra \QQ
\RhG^{(-1)^d}$ is a homomorphism.
\end{cor}

\begin{proof}
From Corollary \ref{lens-spaces-give-all-htpy-types} we have that
for some $k \in  \NN$ there is a homotopy equivalence $f \co
L^{2d-1}(\alpha) \ra L^{2d-1} (\alpha_k)$. It induces a homomorphism
$f_\ast \co \sS^s(L^{2d-1}(\alpha)) \ra \sS^s(L^{2d-1}(\alpha_k))$.
We will show that $\wrho = \wrho \circ f_\ast$. This implies that
$\wrho$ is a homomorphism of $\sS^s(L^{2d-1}(\alpha))$ since it is
equal to a composition of homomorphisms.

We use the observation from Remark \ref{rho-factors} and the
composition formula of \cite[Theorem 2.3]{Ranicki(2008)}. Let $h \co
L \ra L^{2d-1}(\alpha)$ represent an element $a \in \sS^s
(L^{2d-1}(\alpha))$ and note that the homotopy equivalence $f$
represents an element in $\sS^h (L^{2d-1}(\alpha_k))$, call it $b$.
The composition $h \circ f$ represents another element in $\sS^h
(L^{2d-1}(\alpha_k))$, call it $c$. The formula of \cite[Theorem
2.3]{Ranicki(2008)} says $f_\ast a = b-c$. Now clearly
\begin{align*}
\wrho (f_\ast a) & = \wrho (b) -\wrho(c) = \\
 & = \rho (L) - \rho
(L^{2d-1}(\alpha_k)) - \rho(L^{2d-1}(\alpha)) +
\rho(L^{2d-1}(\alpha_k)) = \wrho (a).
\end{align*}
This finishes the proof.
\end{proof}


\section{Calculations} \label{sec:calculation}


We want to prove Theorem \ref{main-thm} by investigating the short
exact sequence (\ref{ses-lens-2d-1}) using the relation to a short
exact sequence from representation theory of $G$ via the
$\rho$-invariant as described in Proposition \ref{ses-vs-rep-thy}.
Theorem \ref{main-thm} is obtained when we put together statements
of Theorems \ref{how-to-split-alpha-k}, \ref{T-2}
 and \ref{T-N}.

\begin{thm} \label{how-to-split-alpha-k}
Let $\bar T = \ker [\wrho] \co \widetilde \sN (L^{2d-1}(\alpha))
\lra \QQ\RhG^{(-1)^d}/4 \cdot \RhG^{(-1)^d}$. Then we have
\[
\sS^s (L^{2d-1}(\alpha)) \cong \bar \Sigma \oplus \bar T
\]
where $\bar \Sigma = \wrho (\sS^s (L^{2d-1}(\alpha)))$ is a free
abelian group of rank $N/2-1$ if $d=2e+1$ and of rank $N/2$ if
$d=2e$.
\end{thm}

\begin{proof} Recall the commutative diagram of
Proposition \ref{ses-vs-rep-thy}. Since $\wrho$ is a homomorphism,
we have that $\bar \Sigma$ is a subgroup of $\QQ \RhG^{(-1)^d}$,
which as an abelian group is a direct sum of $N/2-1$ copies of $\QQ$
if $d=2e+1$ and of $N/2$ copies of $\QQ$ if $d=2e$. It contains a
subgroup $\wrho (\partial \widetilde L^s_{2d} (G))$ which is a free
abelian group of the same rank as the theorem claims for $\bar
\Sigma$ in the respective cases. The claim about the rank of $\bar
\Sigma$ follows.

Now replace in the diagram of Proposition \ref{ses-vs-rep-thy} the
middle and the third term of the lower sequence by the image of
$\wrho$ and by the image of $[\wrho]$ respectively. Then the right
hand square becomes a pullback square. It follows that $\bar T$ is
isomorphic to the kernel of the map $\wrho \co \sS^s
(L^{2d-1}(\alpha)) \lra \bar \Sigma$. We obtain a short exact
sequence of abelian groups
\[
0 \lra \bar T \xra{\lambda} \sS^s (L^{2d-1}(\alpha)) \xra{\wrho}
\bar \Sigma \lra 0
\]
where $\bar \Sigma$ is a free abelian group and hence the sequence
splits.
\end{proof}

So our goal is to understand the subgroup $\bar T$ of $\widetilde
\sN (L^{2d-1}(\alpha_1))$, which is a group isomorphic to the direct
sum $T_N (d) \oplus T_2(d)$ of an $N$-torsion group $T_N(d)$ and a
$2$-torsion group $T_2(d)$
\[
T_N (d) = \bigoplus_{i=1}^{c} \ZZ_N = \bigoplus_{i=1}^{c} \ZZ_N
(t_{4i}) \quad T_2 (d) = \bigoplus_{i=1}^{c} \ZZ_2 =
\bigoplus_{i=1}^{c} \ZZ_2 (t_{4i+2}).
\]
where $c = \lfloor (d-1)/2 \rfloor$. We will denote $\bar T_N (d) =
\bar T \cap T_N (d)$, $\bar T_2 (d) = \bar T \cap T_2 (d)$ and we
will determine the two subgroups separately.

\begin{thm} \label{T-2}
We have
\[
\bar T_2 (d) = T_2 (d)
\]
\end{thm}

\begin{proof}
By Proposition \ref{rho-formula-lens-sp} the formula for $[\wrho]$
only depends on $t_{4i}$.
\end{proof}

\begin{thm} \label{T-N}
We have
\[
\bar T_N (d) = \bigoplus_{i=1}^{c} \ZZ_{2^{\min\{K,2i\}}}
\]
where $c = \lfloor (d-1)/2 \rfloor$.
\end{thm}

In view of Proposition \ref{rho-formula-lens-sp-k} it is convenient
to make the following reformulation. If $d = 2e$ then the group $T_N
(d)$ can be identified with the underlying abelian group $\ZZ_N
[x](d)$ of the truncated polynomial ring in the variable $x$:
\begin{align}
\begin{split} \label{T-N-ident-2e}
T_N (d) & \xra{\cong} \ZZ_N[x](d) := \{ q(x) \in \ZZ_N [x] \; | \;
\deg(q)
\leq c-1 \} \\
t = (t_{4i})_{i=1}^{c} & \mapsto q_t(x) = \sum_{i=0}^{c-1}
t_{4(i+1)} \cdot x^{c-i-1}.
\end{split}
\end{align}
The map $[\wrho]$ becomes
\begin{equation} \label{wrho-formula-2e}
q \mapsto 8 \cdot f'_k \cdot (f^2 -1) \cdot q(f^2).
\end{equation}
If $d = 2e+1$ then the group $T_N (d)$ can be identified with the
underlying abelian group $\ZZ_N [x](d)$ of the truncated polynomial
ring in the variable $x$ as follows:
\begin{align}
\begin{split} \label{T-N-ident-2e+1}
T_N (d) & \xra{\cong} \ZZ_N[x](d) := \{ q(x) \in \ZZ_N [x] \; | \;
\deg(q) \leq c-1 \} \\
t = (t_{4i})_{i=1}^{c} & \mapsto q_t(x) = \sum_{i=1}^{c-1}
t_{4i} \cdot x^{c-i-1}(x-1) + t_{4c} \\
& \quad\quad\quad\; = \sum_{i=1}^{c-1} (t_{4(i+1)}-t_{4i}) x^{c-i-1}
+ t_{4}x^{c-1}
\end{split}
\end{align}
The map $[\wrho]$ then becomes
\begin{equation} \label{wrho-formula-2e+1}
q \mapsto 8 \cdot f_k \cdot q(f^2).
\end{equation}
Further it is convenient to work with the underlying abelian group
of
\[
\ZZ[x](d) := \{ q(x) \in \ZZ [x] \; | \; \deg(q) \leq c-1 \},
\]
use the formulas (\ref{wrho-formula-2e}), (\ref{wrho-formula-2e+1})
to define a map $[\widehat \rho] \colon \ZZ[x](d) \lra \QQ
\RhG^{(-1)^d}$ and study the preimage of $4 \cdot \RhG^{(-1)^d}$. So
the task becomes to find
\begin{align*}
A^k_K(2e) & :=  \left\lbrace q \in \ZZ[x] \mid \deg(q) \leq c-1, \;
8 \cdot f'_k
\cdot (f^2-1) \cdot q(f^2) \in 4 \cdot \ZZ[\chi]/\langle I \langle K \rangle \right\rbrace,\\
A^k_K(2e+1) & := \left\lbrace q \in \ZZ[x] \mid \deg(q) \leq c-1, \;
8 \cdot f_k \cdot q(f^2) \in 4 \cdot \ZZ[\chi]/ I \langle K \rangle
\right\rbrace.
\end{align*}
Here we have replaced $4 \cdot \RhG^{\pm}$ by the bigger $4 \cdot \ZZ[\chi] / \langle K \rangle$. This is legal since when the expressions in question are in $4 \cdot \ZZ[\chi]/I \langle K \rangle$ then they always fulfil the additional conditions to be elements of $4 \cdot \RhG^{\pm}$. We will show that $A^k_K(d) = B_K(d)$ where $B_K(d)$ is a subgroup of polynomials described in terms of certain polynomials
$r^\pm_n(x)$ of degree $n$ for all $n \in \NN$.  These are the best
polynomial of degree $n$ in a sense that
\begin{align*}
8 \cdot f'_k \cdot (f^2-1) \cdot r^+_n(f^2) & \in 4 \cdot \ZZ[\chi]/I \langle 2n+2 \rangle \\
8 \cdot f_k \cdot r^-_n(f^2) & \in 4 \cdot \ZZ[\chi]/I \langle 2n+2
\rangle
\end{align*}
and for all polynomials $q \in \ZZ[x]$ of degree $n$ with leading
coefficient 1 we have
\begin{align*}
8 \cdot f'_k \cdot (f^2-1) \cdot q(f^2) & \notin 4 \cdot \ZZ[\chi]/I \langle 2n+3 \rangle \\
8 \cdot f_k \cdot q(f^2) & \notin 4 \cdot \ZZ[\chi]/ I \langle 2n+3
\rangle.
\end{align*}
We define
\begin{align*}
& B_K(2e) := \left\lbrace \sum_{n=0}^{c-1} a_n \cdot 2^{\max\{K-2n-2,0\}} \cdot r^+_n \mid a_n \in \ZZ \right\rbrace,\\
& B_K(2e+1) := \left\lbrace \sum_{n=0}^{c-1} a_n \cdot
2^{\max\{K-2n-2,0\}} \cdot r^-_n \mid a_n \in \ZZ \right\rbrace.
\end{align*}

\begin{thm} \label{thm:A=B}
\[ A^k_K(d) = B_K(d) \]
\end{thm}

\begin{proof}[Proof of Theorem \ref{T-N}]
It follows from Theorem \ref{thm:A=B} and the definition of $B_K(d)$
that $A_K^k(d)$ is a free abelian subgroup of $\ZZ[x](d)$ with a
basis given by polynomials $r_n^\pm$. Under the homomorphism
$\ZZ[x](d) \ra \ZZ_N[x](d)$ the subgroup $A_K^k (d)$ is mapped onto
a subgroup isomorphic to a direct sum as claimed by the theorem.
\end{proof}

\noindent\textbf{Scheme of the proof of Theorem \ref{thm:A=B}.}

The proof requires a formidable amount of machinery and special
constructions. For better orientation we offer the following scheme.

In subsection \ref{subsec:w-l} we develop some general methods for
deciding whether a given element $g \in \QQ[\chi] / I \langle K
\rangle$ is in $4 \cdot \ZZ[\chi] / I \langle K \rangle$ or not. We
use the Chinese remainder theorem and certain `valuation' functions
$w_l$ on $\QQ[\chi]/I \langle K \rangle$. These are effectively
calculable as is shown in Lemma \ref{calc-rules-w_l}. The criteria
for deciding whether $g$ is in $4 \cdot \ZZ[\chi] / I \langle K
\rangle$ or not using $w_l$ are presented in Theorem
\ref{w_l-theorem}. Later we apply the criteria to decide whether a
polynomial $q \in \ZZ[x](d)$ is in $A_K^k (d)$ or not, but there is
a problem that we do not obtain a necessary and sufficient condition
since the criteria do not apply to all $g \in  \QQ[\chi] / I \langle
K \rangle$. In the sequel we therefore have to combine the
$w_l$-technology with some ad hoc considerations.

In subsection \ref{subsec:good-polys} we construct a sequence of
polynomials in $\ZZ[x](d)$ for $d=2e+1$ which are good in a sense
that they have the leading coefficient $1$ and they yield elements
in $4\cdot \ZZ[\chi] / I \langle K \rangle$ for a large $K$ in
comparison with the other polynomials of the same degree with
leading coefficient $1$. First we construct auxiliary polynomials
$p_k$ in Definition \ref{p-k-defn}, which are used to define `good'
polynomials $q_n$ in Definition \ref{q-n-defn}. These are in turn be
used to define `the best' polynomials $r_n^-$ in Definition
\ref{r_n-defn}. This last definition is inductive, the crucial
inductive step is described in Proposition \ref{r_n-prop}. The
`goodness' properties are summarized in Corollary \ref{p_k-w_l}
($p_k$), in Proposition \ref{K(q-n)-prop} ($q_n$) and in Corollary
\ref{r_n-properties} ($r^-_n$).

The final subsection \ref{subsec:A=B} completes the proof of Theorem
\ref{thm:A=B}. This part treats first the case $d = 2e+1$ and
proceeds by induction on $K$. The proof in the case $d=2e$ is short
and proceeds by a reduction to the case $d=2e+1$.

\subsection{$w_l$ technology} \label{subsec:w-l}

\

\medskip

For given $g \in \QQ\RhG$ we want to decide whether $g \in 4 \cdot
\RhG$ or not using the homomorphisms $\pr_l: \QQ\RhG \cong \QQ[\chi]
/ I \langle K \rangle \twoheadrightarrow \QQ[\chi] / \langle
1+\chi^{2^l} \rangle$ for $0 \leq l \leq K-1$. Obviously, $g \in 4
\cdot \RhG$ implies $\pr_l(g) \in 4 \cdot \ZZ[\chi] / \langle
1+\chi^{2^l} \rangle$. Of more interest is the other direction. By
the Chinese remainder theorem $g$ is uniquely determined by the
elements $\pr_l(g)$ ($0 \leq l \leq K-1$). More precisely, we have

\begin{lem}\label{Chinese remainder-trick}
Let $g \in \QQ\RhG$. Then
\[
g = \sum_{l=0}^{K-1} 2^{l-K} \cdot g_l \cdot (1-\chi) \cdot
\prod_{\stackrel{0 \leq r \leq K-1}{r \neq l}} (1+\chi^{2^r})
\]
for any elements $g_l \in \QQ\RhG$ satisfying $\pr_l(g_l) =
\pr_l(g)$.
\end{lem}

If $\pr_l(g) \in 2^{2+K-l} \cdot \ZZ[\chi] / \langle 1+\chi^{2^l}
\rangle$ we can choose $g_l \in 2^{2+K-l} \cdot \RhG$ satisfying
$\pr_l(g_l) = \pr_l(g)$ and the lemma above shows $g \in 4 \cdot
\RhG$. Motivated by this observation we want to analyze whether
$\pr_l(g)$ lies in $2^m \cdot \ZZ[\chi] / \langle 1+\chi^{2^l}
\rangle$ for some integer $m$. For this purpose we will introduce
$w_l$-functions which are generalizations of the p-adic valuation
for $p=2$.

Before we do so, we give the proof of Lemma \ref{Chinese
remainder-trick} and consider an application (Lemma
\ref{injectivity-mult-f}).

\begin{proof}[Proof of Proposition \ref{Chinese remainder-trick}]
We have $\QQ\RhG \cong \QQ[\chi] / I \langle K \rangle$ where $I
\langle K \rangle$ was defined as $I \langle K \rangle := \langle
1+\chi+\cdots+\chi^{2^K-1} \rangle$. Notice that
$1+\chi+\cdots+\chi^{2^K-1} = \prod_{m=0}^{K-1} (1+\chi^{2^m})$.
Since the factors $1+\chi^{2^m}$ are mutually coprime in the
principal ideal domain $\QQ[\chi]$, it suffices to check the desired
equality under the epimorphism $\pr_m$ for $0 \leq m \leq K-1$. In
$\QQ[\chi] / \langle 1+\chi^{2^m} \rangle$ we obtain
\begin{align*}
& \pr_m \left( \sum_{l=0}^{K-1} 2^{l-K} \cdot g_l \cdot (1-\chi) \cdot \prod_{\stackrel{0 \leq r \leq K-1}{r \neq l}} (1+\chi^{2^r}) \right) & =\\
& \sum_{l=0}^{K-1} 2^{l-K} \cdot \pr_m(g) \cdot (1-\chi) \cdot \prod_{\stackrel{0 \leq r \leq K-1}{r \neq l}} (1+\chi^{2^r}) & =\\
& 2^{m-K} \cdot \pr_m(g) \cdot (1-\chi) \cdot \prod_{\stackrel{0 \leq r \leq K-1}{r \neq m}} (1+\chi^{2^r}) & =\\
& 2^{m-K} \cdot \pr_m(g) \cdot (1-\chi) \cdot (1+\chi+\cdots+\chi^{2^m-1}) \cdot \prod_{r=m+1}^{K-1} (1+(-1)^{2^{r-m}}) & =\\
& 2^{m-K} \cdot \pr_m(g) \cdot (1-\chi^{2^m}) \cdot 2^{K-1-m} \qquad
= \qquad \pr_m(g). &
\end{align*}
\end{proof}

As a warm-up in learning how to work with Lemma \ref{Chinese
remainder-trick} we prove the following lemma needed in the proof of
Proposition \ref{rho-homomorphism}.

\begin{lem} \label{injectivity-mult-f}
Let $z \in \QQ \RhG^-$. If $f \cdot z = 0$ in $\QQ \RhG^+$ then $z =
0$.
\end{lem}
\begin{proof}
It follows from Lemma \ref{Chinese remainder-trick} that it is
sufficient to show $\pr_l(z) = 0$ for all $0 \leq l \leq K-1$. We
have $\pr_l(f) \cdot pr_l(z) = pr_l(f \cdot z) = 0$. Notice that
$\pr_l(f)$ is invertible for $l \geq 1$ since
\[
(1+\chi)^{-1} = \frac{1}{2} \cdot (1-\chi+\chi^2-\chi^3+ \cdots
-\chi^{2^l-1}) \in \QQ[\chi]/\langle 1 + \chi^{2^l} \rangle.
\]
This implies $pr_l(z)=0$ for $l \geq 1$. Further recall that we can
write $z \in \QQ \RhG^-$ as
\[
z = \sum_{r=1}^{N/2-1} a_r \cdot (\chi^r - \chi^{N-r}).
\]
with $a_r \in \QQ$. Since $\chi^r - \chi^{N-r}$ is a multiple of $1
+ \chi$, we conclude $\pr_0 (z) = 0$.
\end{proof}

The following lemma is needed for the definition of the
$w_l$-functions.

\begin{lem}\label{preparation-w_l}
Let $g \in \QQ\RhG$ and $l \geq 0$ such that $(1+\chi^{2^l}) \nmid
g$.
\begin{enumerate}
 \item There exist $a, u \in \ZZ$, $b \in \{0,1,\cdots,2^l-1\}$ and $v_1, v_2 \in \ZZ[\chi]$ such that
\begin{itemize}
 \item $\pr_l(g)$ and $\frac{2^a}{u} \cdot \left( (1-\chi)^b \cdot v_1(\chi) + 2 \cdot v_2(\chi) \right)$
coincide in $\QQ[\chi] / \langle 1+\chi^{2^l} \rangle$,
\item $u$ and $v_1(1)$ are odd.
\end{itemize}
Moreover, if $\pr_l(g)\in \ZZ[\chi] / \langle 1+\chi^{2^l} \rangle$
then $u$ can be chosen to be 1.
 \item The numbers $a$ and $b$ are uniquely determined by $g$ and $l$.
\end{enumerate}
\end{lem}
\begin{proof}
(1) There exists $w \in \ZZ$ such that $w \cdot \pr_l(g) \in
\ZZ[\chi] / \langle 1+\chi^{2^l} \rangle$. We write $w = 2^{a_1}
\cdot u$ with $u$ odd. Choose $z \in \ZZ[\chi]$ of degree $\deg(z) <
2^l$ such that $w \cdot \pr_l(g)$ and $z(\chi)$ coincide in
$\QQ[\chi] / \langle 1+\chi^{2^l} \rangle$. We write $z$ as $z(\chi)
= \sum_{m=0}^{2^l-1} z_m \cdot (1-\chi)^m$ with $z_m \in \ZZ$. Since
$(1+\chi^{2^l})$ does not divide $g$, we have $z \neq 0$. Hence
there exists a largest number $a_2 \in \NN_0$ such that
$\frac{z}{2^{a_2}} \in \ZZ[\chi]$. Define $a:=-a_1+a_2$,
$b:=\min\left\lbrace m \; \big\vert \; 2 \nmid
\frac{z_m}{2^{a_2}}\right\rbrace$, $v_1(\chi):=\sum_{m=b}^{2^l-1}
\frac{z_m}{2^{a_2}} \cdot (1-\chi)^{m-b}$ and $v_2(\chi) :=
\sum_{m=0}^{b-1} \frac{z_m}{2^{a_2+1}} \cdot (1-\chi)^m$. Then
$v_1(1) = \frac{z_b}{2^{a_2}}$ is odd and $\pr_l(g)$ coincides with
\[
\frac{2^a}{u} \cdot \left( (1-\chi)^b \cdot v_1(\chi) + 2 \cdot
v_2(\chi) \right)
\]
in $\QQ[\chi] / \langle 1+\chi^{2^l} \rangle$.

(2) Let $a, a', u, u' \in \ZZ$, $b, b' \in \{0,1,\cdots,l-1\}$ and
$v_1, v'_1, v_2, v'_2 \in \ZZ[\chi]$ such that $u, u', v_1(1),
v'_1(1)$ are odd and
\begin{align*}
\pr_l(g) & = \frac{2^a}{u} \cdot \left( (1-\chi)^b \cdot v_1(\chi) + 2 \cdot v_2(\chi) \right)\\
& = \frac{2^{a'}}{u'} \cdot \left( (1-\chi)^{b'} \cdot v'_1(\chi) +
2 \cdot v'_2(\chi) \right)
\end{align*}
in $\QQ[\chi] / \langle 1+\chi^{2^l} \rangle$.

We first prove $a = a'$ by contradiction. Assume that $a < a'$. In
$\QQ[\chi] / \langle 1+\chi^{2^l} \rangle$ we get
\[
u' \cdot \left( (1-\chi)^b \cdot v_1(\chi) + 2 \cdot v_2(\chi)
\right) = 2^{a'-a} \cdot u \cdot \left( (1-\chi)^{b'} \cdot
v'_1(\chi) + 2 \cdot v'_2(\chi) \right).
\]
Hence there exists $q \in \QQ[\chi]$ with
\begin{align*}
& u' \cdot \left( (1-\chi)^b \cdot v_1(\chi) + 2 \cdot v_2(\chi) \right) =\\
& 2^{a'-a} \cdot u \cdot \left( (1-\chi)^{b'} \cdot v'_1(\chi) + 2
\cdot v'_2(\chi) \right) + q(\chi) \cdot (1+\chi^{2^l}).
\end{align*}
The equation above implies $q(\chi) \cdot (1+\chi^{2^l}) \in
\ZZ[\chi]$ and hence $q \in \ZZ[\chi]$. Under the epimorphism
$\ZZ[\chi] \twoheadrightarrow \ZZ_2[\chi]$ this equation becomes to
\[
(1+\chi)^b \cdot \overline{v_1}(\chi) = \overline{q}(\chi) \cdot
(1+\chi^{2^l}).
\]
Since $1+\chi^{2^l} = (1+\chi)^{2^l}$ in $\ZZ_2[\chi]$ (proof by
induction), we conclude $\overline{v_1}(\chi) = \overline{q}(\chi)
\cdot (1+\chi)^{2^l-b}$ and hence $\overline{v_1}(1) = 0$ in
$\ZZ_2[\chi]$. This is a contradiction to $v_1(1)$ odd. Therefore,
we have $a = a'$.

It remains to prove $b = b'$. We give again a proof by
contradiction. Assume that $b < b'$. In $\QQ[\chi] / \langle
1+\chi^{2^l} \rangle$ we have
\[
u' \cdot \left( (1-\chi)^b \cdot v_1(\chi) + 2 \cdot v_2(\chi)
\right) = u \cdot \left( (1-\chi)^{b'} \cdot v'_1(\chi) + 2 \cdot
v'_2(\chi) \right).
\]
Hence there exists $q \in \QQ[\chi]$ with
\begin{align*}
& u' \cdot \left( (1-\chi)^b \cdot v_1(\chi) + 2 \cdot v_2(\chi) \right) =\\
& u \cdot \left( (1-\chi)^{b'} \cdot v'_1(\chi) + 2 \cdot v'_2(\chi)
\right) + q(\chi) \cdot (1+\chi^{2^l}).
\end{align*}
We conclude $q \in \ZZ[\chi]$. Under the epimorphism $\ZZ[\chi]
\twoheadrightarrow \ZZ_2[\chi]$ this equation becomes to
\[
(1+\chi)^b \cdot \overline{v_1}(\chi) = (1+\chi)^{b'} \cdot
\overline{v'_1}(\chi) + \overline{q}(\chi) \cdot (1+\chi^{2^l}).
\]
We finally get $\overline{v_1}(\chi) = (1+\chi)^{b'-b} \cdot
\overline{v'_1}(\chi) + \overline{q}(\chi) \cdot (1+\chi)^{2^l-b}$
and hence $\overline{v_1}(1) = 0$ in $\ZZ_2[\chi]$ contradicting
$v_1(1)$ odd. This shows $b = b'$.
\end{proof}

\begin{defn}\label{def-w_l}
For $l \geq 0$ define
\begin{align*}
w_l & \colon \QQ\RhG \lra \frac{1}{2^l} \cdot \ZZ \cup \{\infty\} \\
& g \mapsto
\begin{cases} a+b/2^l \quad \mathrm{with \;} a, b \mathrm{\; as \; in \; Lemma \; \ref{preparation-w_l} \; when \;} (1+\chi^{2^l}) \nmid g\\
\infty \quad \mathrm{when \;} (1+\chi^{2^l}) \mid g
\end{cases}
\end{align*}
\end{defn}

Roughly speaking, $w_l$ counts how many factors of $2$ are contained
in $\pr_l(g)$. We have the following calculation rules.

\begin{lem}\label{calc-rules-w_l}
Let $g_1, g_2 \in \QQ\RhG$ and $l \geq 0$.
\begin{enumerate}
 \item $w_l(g_1 \cdot g_2) = w_l(g_1) + w_l(g_2)$.
 \item If $w_l(g_1) \neq w_l(g_2)$ then $w_l(g_1 + g_2) = \min\left\lbrace  w_l(g_1), w_l(g_2) \right\rbrace$.
 \item If $w_l(g_1) = w_l(g_2)$ then $w_l(g_1 + g_2) > w_l(g_1) = w_l(g_2)$.
\end{enumerate}
\end{lem}
\begin{proof}
Parts (2) and (3) can be proven by easy calculations. We only focus
on (1). As in Lemma \ref{preparation-w_l} we write (for $i=1,2$):
\[
\pr_l(g_i) = \frac{2^{a_i}}{u_i} \cdot \left( (1-\chi)^{b_i} \cdot
v_1^{(i)}(\chi) + 2 \cdot v_2^{(i)}(\chi) \right) \in \QQ[\chi] /
\langle 1+\chi^{2^l} \rangle.
\]
In $\QQ[\chi] / \langle 1+\chi^{2^l} \rangle$ we obtain
\begin{align*}
& \pr_l(g_1 \cdot g_2) = \pr_l(g_1) \cdot \pr_1(g_2) =\\
& \frac{2^{a_1+a_2}}{u_1 \cdot u_2} \cdot (1-\chi)^{b_1+b_2} \cdot v_1^{(1)}(\chi) \cdot v_1^{(2)}(\chi) +\\
& \frac{2^{a_1+a_2+1}}{u_1 \cdot u_2} \cdot \left( (1-\chi)^{b_1}
\cdot v_2^{(2)}(\chi) + (1-\chi)^{b_2} \cdot v_2^{(1)}(\chi) + 2
\cdot v_2^{(1)}(\chi) \cdot v_2^{(2)}(\chi) \right).
\end{align*}
If $b_1 + b_2 < 2^l$ then we conclude
\[
w_l(g_1 \cdot g_2) = a_1+a_2 + \frac{b_1+b_2}{2^l} = a_1 +
\frac{b_1}{2^l} + a_2 + \frac{b_2}{2^l} = w_l(g_1) + w_l(g_2).
\]

It remains to study the case $b_1 + b_2 \geq 2^l$. Define $x_m \in
\ZZ[\chi]$ by $x_0 := -\chi$ and $x_m := \chi^{2^{m-1}} + 2 \cdot
x_{m-1}(\chi) \cdot (1+\chi^{2^{m-1}} + x_{m-1}(\chi))$ for $m \geq
1$. A proof by induction shows $(1-\chi)^{2^k} = 2 \cdot x_k(\chi) +
(1+\chi^{2^k})$ for all $k \geq 0$. In $\QQ[\chi] / \langle
1+\chi^{2^l} \rangle$ we obtain
\begin{align*}
& \pr_l(g_1 \cdot g_2) = \pr_l(g_1) \cdot \pr_1(g_2) = \frac{2^{a_1+a_2+1}}{u_1 \cdot u_2} \cdot (1-\chi)^{b_1+b_2-2^l} \cdot\\
& \quad \left( x_l(\chi) \cdot v_1^{(1)}(\chi) \cdot v_1^{(2)}(\chi) + (1-\chi)^{2^l-b_2} \cdot v_2^{(2)}(\chi) + (1-\chi)^{2^l-b_1} \cdot v_2^{(1)}(\chi) \right) +\\
& \quad \frac{2^{a_1+a_2+2}}{u_1 \cdot u_2} \cdot v_2^{(1)}(\chi)
\cdot v_2^{(2)}(\chi).
\end{align*}
We finally conclude
\[
w_l(g_1 \cdot g_2) = a_1+a_2+1 + \frac{b_1+b_2-2^l}{2^l} = a_1 +
\frac{b_1}{2^l} + a_2 + \frac{b_2}{2^l} = w_l(g_1) + w_l(g_2).
\]
\end{proof}

\begin{rem}
For the reader who is familiar with number fields and valuations we
give the following description for the $w_l$-functions.

Let $\zeta_{2^{l+1}} \in \CC$ be a primitive $2^{l+1}$-th root of
unity. Consider the ring of algebraic integers
$\ZZ[\zeta_{2^{l+1}}]$ in the cyclotomic field
$\QQ(\zeta_{2^{l+1}})$. The ideal ${\mathcal P}
:=(2,1-\zeta_{2^{l+1}})$ in $\ZZ[\zeta_{2^{l+1}}]$ is a prime ideal
satisfying ${\mathcal P}^{2^l} = (2)$. Let $\nu_{\mathcal P}$ be the
(exponential) valuation with respect to this prime ideal ${\mathcal
P}$. Then the $w_l$-function is given by
\[
w_l(g) = \frac{1}{2^l} \cdot \nu_{\mathcal P} \left( \alpha \left(
\pr_l(g) \right) \right)
\]
where $\alpha: \QQ[\chi] / \langle 1+\chi^{2^l} \rangle \to
\QQ(\zeta_{2^{l+1}})$ is the isomorphism induced by $\chi \mapsto
\zeta_{2^{l+1}}$.
\end{rem}

\begin{expl} \label{w-l-expls}
The reader is invited to calculate the $w_l$ for the following
examples.
\begin{enumerate}
\item If $q \in \QQ \subset \QQ\RhG$ then $w_l(q)$ coincides with the p-adic valuation for $p=2$. In particular, $w_l(2^a)=a$.
\item $w_l(f) = \begin{cases} \infty \quad \mathrm{when \;} l=0\\ 0 \quad \mathrm{when \;} l \geq 1\end{cases}$\\
\item $w_l(f \pm 1) = 1-2^{-l}$
\item $w_l(f^2-1) = 2-2^{1-l}$
\item $w_l(f^2+1) = \begin{cases} 0 \quad \mathrm{when \;} l=0\\ \infty \quad \mathrm{when \;} l=1\\ 1 \quad \mathrm{when \;} l \geq 2\end{cases}$
\item $w_l(f'_k) = 0$
\end{enumerate}
Hints. (2): For $l=0$ use $f \equiv 0 \; \mod 1+\chi$. For $l \geq
1$ use $f \cdot (1-\chi) = 1+\chi$ and $1+\chi = 2-(1-\chi)$. (3):
Use $(1-\chi)(f+1) = 2$ and $(1-\chi)(f-1) = 2\chi$. (4): Use (3).
(5): For $l=1$ use $f^2+1 \equiv 0$ $\mod 1+\chi^2$. For $l \neq 1$
use $f^2+1 = (f^2-1)+2$ and (4). (6): Use $f'_k$, ${f'_k}^{-1} \in
\RhG$ and the fact $w_l (g) \geq 0$ when $g \in \RhG$.
\end{expl}

We now come back to the initial question of this subsection: For a
given $g \in \QQ\RhG$ we want to decide whether $g$ lies in $4 \cdot
\RhG$ or not. The following theorem can answer this question in many
cases.

\begin{thm}\label{w_l-theorem}
Let $g \in \QQ\RhG$. Suppose that $\pr_l(g) \in 4 \cdot \ZZ[\chi] /
\langle 1+\chi^{2^l} \rangle$ for all $0 \leq l \leq K-1$.
\begin{enumerate}
 \item If $w_l(g) \geq 2+K-l-2^{-l}$ for all $0 \leq l \leq K-1$ then $g \in 4 \cdot \RhG$.
 \item If there exist $h \in \RhG$ and $0 \leq l' \leq K-1$ such that
\begin{align*}
 & w_l(g) + w_l(h) \geq 2+K-l-2^{-l} \mathrm{\; for \; all \;} l \in \{0,1,\cdots,K-1\}-\{l'\} \mathrm{\; and}\\
 & w_{l'}(g) + w_{l'}(h) < 2+K-l'-2^{-l'}
\end{align*}
 then $g \notin 4 \cdot \RhG$.
\end{enumerate}
\end{thm}
\begin{proof}
(1) The assumption $w_l(g) \geq 2+K-l-2^{-l}$ implies $w_l((1-\chi)
\cdot g) \geq 2+K-l$. By Lemma \ref{preparation-w_l} and Definition
\ref{def-w_l} there exist $z_l \in \ZZ[\chi]$ such that
\[
(1-\chi) \cdot \pr_l(g) = 2^{2+K-l} \cdot z_l(\chi) \in \QQ[\chi] /
\langle 1+\chi^{2^l} \rangle
\]
for all $0 \leq l \leq K-1$. Using Lemma \ref{Chinese
remainder-trick} we conclude
\[
(1-\chi) \cdot g = \sum_{l=0}^{K-1} 2^{l-K} \cdot \left( 2^{2+K-l}
\cdot z_l(\chi) \right) \cdot (1-\chi) \cdot \prod_{\stackrel{0 \leq
r \leq K-1}{r \neq l}} (1+\chi^{2^r}) \mathrm{\quad in \;} \QQ\RhG
\]
and hence
\[
g = 4 \cdot \sum_{l=0}^{K-1}  \cdot z_l(\chi) \cdot
\prod_{\stackrel{0 \leq r \leq K-1}{r \neq l}} (1+\chi^{2^r}) \in 4
\cdot \RhG.
\]
(2) We give a proof by contradiction. Assume that $g \in 4 \cdot
\RhG$ and define
\[
a:=\min\left\lbrace m \in \ZZ \; \vert \; m + w_{l'}(g) + w_{l'}(h)
\geq 2+K-l'-2^{-l'} \right\rbrace.
\]
Notice that $a \geq 1$. We have
\[
w_l((1-\chi) \cdot 2^a \cdot g \cdot h) \geq 3+K-l \mathrm{\quad for
\; all \;} l \in \{0,1,\cdots,K-1\}-\{l'\}
\]
and
\[
w_{l'}((1-\chi) \cdot 2^a \cdot g \cdot h) \geq 2+K-l'.
\]
From Lemma \ref{preparation-w_l} and Definition \ref{def-w_l} we
conclude that there exist $z_l \in 2 \cdot \ZZ[\chi]$ for all $l \in
\{0,1,\cdots,K-1\}-\{l'\}$ and $z_{l'} \in \ZZ[\chi]$ satisfying
\[
\pr_l((1-\chi) \cdot 2^a \cdot g \cdot h) = 2^{2+K-l} \cdot
z_l(\chi) \in \QQ[\chi] / \langle 1+\chi^{2^l} \rangle.
\]
Lemma \ref{Chinese remainder-trick} implies
\[
(1-\chi) \cdot 2^a \cdot g \cdot h = \sum_{l=0}^{K-1} 2^{l-K} \cdot
2^{2+K-l} \cdot z_l(\chi) \cdot (1-\chi) \cdot \prod_{\stackrel{0
\leq r \leq K-1}{r \neq l}} (1+\chi^{2^r})
\]
and hence
\[
2^a \cdot g \cdot h = \sum_{l=0}^{K-1} 4 \cdot z_l(\chi) \cdot
\prod_{\stackrel{0 \leq r \leq K-1}{r \neq l}} (1+\chi^{2^r})
\mathrm{\quad in \; \QQ\RhG}.
\]
Since $g \cdot h \in 4 \cdot \RhG$ there exists $y \in \ZZ[\chi]$
such that $g \cdot h$ and $4 \cdot y$ coincide in $\QQ[\chi] / I
\langle K \rangle$. We get
\[
2^a \cdot y (\chi) = \sum_{l=0}^{K-1} z_l(\chi) \cdot
\prod_{\stackrel{0 \leq r \leq K-1}{r \neq l}} (1+\chi^{2^r})
\mathrm{\quad in \;} \QQ[\chi] / I \langle K \rangle.
\]
Hence there exists $q \in \QQ[\chi]$ with
\[
2^a \cdot y (\chi) = \sum_{l=0}^{K-1} z_l(\chi) \cdot
\prod_{\stackrel{0 \leq r \leq K-1}{r \neq l}} (1+\chi^{2^r}) +
q(\chi) \cdot (1+\chi+\cdots+\chi^{N-1}).
\]
The equation above implies $q(\chi) \cdot (1+\chi+\cdots+\chi^{N-1})
\in \ZZ[\chi]$ and hence $q \in \ZZ[\chi]$. Under the epimorphism
$\ZZ[\chi] \twoheadrightarrow \ZZ_2[\chi]$ this equation becomes
\begin{align*}
0 = & \sum_{l=0}^{K-1} \overline{z_l}(\chi) \cdot \prod_{\stackrel{0 \leq r \leq K-1}{r \neq l}} (1+\chi^{2^r}) + \overline{q}(\chi) \cdot (1+\chi+\cdots+\chi^{N-1})\\
= & \overline{z_{l'}}(\chi) \cdot \prod_{\stackrel{0 \leq r \leq
K-1}{r \neq l'}} (1+\chi^{2^r}) + \overline{q}(\chi) \cdot
\prod_{r=0}^{K-1} (1+\chi^{2^r}).
\end{align*}
Hence $\overline{z_{l'}}(\chi) = - \overline{q}(\chi) \cdot
(1+\chi^{2^{l'}}) = - \overline{q}(\chi) \cdot (1+\chi)^{2^{l'}}$ in
$\ZZ_2[\chi]$. This implies $w_{l'}(z_{l'}) \geq 1$. We finally get
\begin{align*}
& (a-1) + w_{l'}(g) + w_{l'}(h) & =\\
& w_{l'}((1-\chi) \cdot 2^a \cdot g \cdot h) - 1 - 2^{-l'} & =\\
& w_{l'}(2^{2+K-l'} \cdot z_l'(\chi)) - 1 - 2^{-l'} & =\\
& w_{l'}(z_l'(\chi)) + 1 + K - l' - 2^{-l'} & \geq\\
& 2 + K - l' - 2^{-l'}
\end{align*}
which contradicts the minimality of $a$.
\end{proof}

\subsection{Good polynomials} \label{subsec:good-polys}

\

\medskip

Theorem \ref{w_l-theorem} can be used to decide for a given $K \in
\NN$ whether the expression $8 \cdot f'_k \cdot f \cdot q(f^2) \in
\QQ [\chi] / I \langle K \rangle$ lands in $4 \cdot \ZZ[\chi] / I
\langle K \rangle$ or not for many but not all polynomials $q(x) \in
\QQ[x]$. Here we introduce polynomials $r_n^-$, which are in the
next subsection proved to be the best in a sense that they are
polynomials with leading coefficient $1$ yielding elements in $4
\cdot \ZZ[\chi] / I \langle K \rangle$ for a large $K$ in comparison
with the other polynomials of the same degree with leading
coefficient $1$.

We start by defining auxiliary polynomials $p_k$ which are used to
define polynomials denoted $q_n$ whose properties are summarized in
Proposition \ref{K(q-n)-prop}. A careful analysis gives an inductive
procedure for construction of certain linear combinations of
polynomials $q_n$, denoted $\widetilde q_n$, with better properties
that $q_n$ themselves. This is the content of Proposition
\ref{r_n-prop}. The assumptions of this proposition turn out to be
trivially true for small $n$ which enables us to perform the
induction to define the desired polynomials $r_n^-$ in Definition
\ref{r_n-defn}.

Notice that the for any $q(x) \in \QQ[x]$ we get
\[
w_0 (8 \cdot f'_k \cdot f \cdot q(f^2))=\infty
\]
since $w_0 (f) = \infty$ because of ($1+\chi) \divides f$. Further
notice that for $p_1(x) := x+1$ we have
\begin{equation} \label{p-1}
p_1(f^2) = f^2 +1 = 2 \cdot \frac{1+\chi^2}{(1-\chi)^2}.
\end{equation}
Hence $(1+\chi^2) \divides p_1(f^2)$ in $\QQ[\chi] / I \langle K
\rangle$ and $w_1 (8 \cdot f'_k \cdot f \cdot p_1(f^2))=\infty$.
Further observe that
\begin{equation} \label{substitution}
\frac{(f^2+1)^2}{4 \cdot f^2} = \frac{(1+\chi^2)^2}{(1-\chi^2)^2}.
\end{equation}
Motivated by that we make the following
\begin{defn} \label{p-k-defn}
Let $p_1 (x) := x+1 \in \ZZ[x]$. For $k \in  \NN$ define inductively
\begin{equation} \label{p-ks}
p_{k+1} (x) := p_k \bigg( \frac{(x+1)^2}{4x} \bigg) \cdot
(4x)^{2^{k-1}} \in \ZZ[x].
\end{equation}
\end{defn}

Notice that $p_k (x)$ is a polynomial in $\ZZ[x]$ of degree
$2^{k-1}$.

\begin{thm} \label{p-k-thm}
We have
\[
p_k (f^2) = 2^{2^k-1} \cdot \frac{1+\chi^{2^k}}{(1-\chi)^{2^k}} \in
\QQ[\chi] / I \langle K \rangle \quad  \textup{for} \; k \in \NN.
\]
\end{thm}

\begin{proof}
It suffices to prove the equality
\[
p_k (f^2) = 2^{2^k-1} \cdot \frac{1+\chi^{2^k}}{(1-\chi)^{2^k}}
\]
in the field of rational functions $\QQ(\chi)$. The proof now goes
by induction with respect to $k \in \NN$. The case $k=1$ is proved
by the identity (\ref{p-1}). Now the induction step. Let $\alpha \co
\QQ(\chi) \lra \QQ(\chi)$ be the homomorphism given by $\chi \mapsto
\chi^2$. We calculate:
\begin{align*}
p_{k+1} (f^2) & = p_k \bigg( \frac{(f^2+1)^2}{4 \cdot f^2} \bigg) \cdot \big(4f^2\big)^{2^{k-1}} \\
& = p_k \bigg(\bigg( \frac{1+\chi^2}{1-\chi^2} \bigg)^2\bigg) \cdot 2^{2^k} \cdot f^{2^k} \; \quad \quad \quad \textup{by (\ref{substitution})} \\
& = p_k \big((\alpha (f))^2\big) \cdot 2^{2^k} \cdot f^{2^k} \\
& = \alpha \big(p_k (f^2)\big) \cdot 2^{2^k} \cdot f^{2^k} \\
& = \frac{1+(\chi^2)^{2^k}}{(1-\chi^2)^{2^k}} \cdot 2^{2^k-1} \cdot 2^{2^k} \cdot \frac{(1+\chi)^{2^k}}{(1-\chi)^{2^k}} \\
& = \frac{1+\chi^{2^{k+1}}}{(1-\chi)^{2^k} \cdot (1+\chi)^{2^k}} \cdot 2^{2^{k+1}-1} \cdot \frac{(1+\chi)^{2^k}}{(1-\chi)^{2^k}} \\
& = 2^{2^{k+1}-1} \cdot \frac{1+\chi^{2^{k+1}}}{(1-\chi)^{2^{k+1}}}
\end{align*}
\end{proof}

\begin{cor} \label{p_k-w_l}
We have
\begin{enumerate}
\item $w_l (p_k (f^2)) = \infty$ when $l=k$
\item $w_l (p_k (f^2)) = 2^k-1$ when $l>k$
\end{enumerate}
\end{cor}

\begin{proof}
The first item is immediate from the formula of the previous
theorem. For the second item note that $1+\chi^{2^k} \equiv
(1-\chi)^{2^k} \mod 2$. It follows that $w_l (1+\chi^{2^k}) = w_l
((1-\chi)^{2^k})$ for $l > k$ and hence
\[
w_l \bigg(\frac{1+\chi^{2^k}}{(1-\chi)^{2^k}}\bigg) = 0 \quad
\textup{for} \; l>k.
\]
Finally use the formula of the previous theorem and the product
formula for $w_l$.
\end{proof}

Now we are ready to introduce the polynomials $q_n$ which will be
good in the already mentioned sense. The idea is that we get good
polynomials when we multiply the polynomials $p_k$ from the previous
definition.

\begin{defn} \label{q-n-defn}
Let $n \geq 0$. Define $a(n), b(n) \geq 0$ as the integers
satisfying
\[
n+1 = 2^{a(n)} + b(n) \mathrm{\quad with \quad} 0 \leq b(n) \leq
2^{a(n)}-1.
\]
Define
\[
q_n (x) := \prod_{r=1}^{a(n)} p_r (x) \cdot (x-1)^{b(n)}.
\]
\end{defn}

\begin{prop} \label{K(q-n)-prop}
Let $n \geq 0$, $k \geq 1$ and $m \in \{1,2\}$. We have
\begin{align*}
& 8 \cdot f'_k \cdot f^m \cdot q_n(f^2) \in 4 \cdot \ZZ[\chi]/ I \langle 2n+1 \rangle,\\
& 8 \cdot f'_k \cdot f^m \cdot q_n(f^2) \in 4 \cdot \ZZ[\chi]/ I \langle 2n+2 \rangle \Longleftrightarrow b(n) = 0,\\
& 8 \cdot f'_k \cdot f^m \cdot q_n(f^2) \notin 4 \cdot \ZZ[\chi]/ I
\langle 2n+3 \rangle.
\end{align*}
Moreover, we have
\begin{align*}
& 8 \cdot f'_k \cdot f^m \cdot q_n(f^2) \cdot (1-\chi)^{2b(n)-1} \in 4 \cdot \ZZ[\chi]/ I \langle 2n+2 \rangle \mathrm{\; if \;} b(n)>0,\\
& 8 \cdot f'_k \cdot f^m \cdot q_n(f^2) \cdot (1-\chi)^{2n+1+2^{a(n)} (2^s-2)} \in 4 \cdot \ZZ[\chi]/ I \langle 2n+2+s \rangle \mathrm{\; for \; all \;} s \geq 1,\\
& 8 \cdot f'_k \cdot f^m \cdot q_n(f^2) \cdot (1-\chi)^{2n} \notin 4
\cdot \ZZ[\chi]/ I \langle 2n+3 \rangle.
\end{align*}
\end{prop}

We use the $w_l$-technology for which we need:

\begin{lem} \label{w_l(q_n)}
Let $n \geq 0$, $k \geq 1$ and $m \in \{1,2\}$. We have
\[
w_l(8 \cdot f'_k \cdot f^m \cdot q_n(f^2)) = \begin{cases} \infty  &
l \leq a(n) \\ 2n+3-a(n)-\frac{b(n)}{2^{l-1}} & l \geq a(n)+1
\end{cases}
\]
\end{lem}
\begin{proof}
Use the formulas from Lemma \ref{calc-rules-w_l}, Example
\ref{w-l-expls} and Corollary \ref{p_k-w_l}.
\end{proof}

\begin{proof}[Proof of Proposition \ref{K(q-n)-prop}]
The desired results are obtained using the criteria from Theorem
\ref{w_l-theorem}.
\begin{align*}
& w_l \left( 8 \cdot f'_k \cdot f^m \cdot q_n(f^2) \right) - \left( 2+2n+1-l-2^{-l} \right) =\\
& \qquad \qquad \begin{cases} \infty & \; l \leq a(n) \\
l-a(n)-\frac{2b(n)-1}{2^l} \geq 0 & \; l \geq a(n)+1 \end{cases}
\end{align*}
implies
\[
8 \cdot f'_k \cdot f^m \cdot q_n(f^2) \in 4 \cdot \ZZ[\chi]/ I
\langle 2n+1 \rangle.
\]
For $b(n) > 0$ we have
\begin{align*}
& w_l \left( 8 \cdot f'_k \cdot f^m \cdot q_n(f^2) \right) - \left( 2+2n+2-l-2^{-l} \right) =\\
& \qquad \qquad \begin{cases} \infty & \; l \leq a(n) \\
-\frac{2b(n)-1}{2^{a(n)+1}} < 0 & \; l = a(n)+1 \\
l-(a(n)+1)-\frac{2b(n)-1}{2^l} \geq 0 &  \; l \geq a(n)+2
\end{cases}
\end{align*}
and
\begin{align*}
& w_l \left( 8 \cdot f'_k \cdot f^m \cdot q_n(f^2) \cdot (1-\chi)^{2b(n)-1} \right) - \left( 2+2n+2-l-2^{-l} \right) =\\
& \qquad \qquad \begin{cases} \infty & \; l \leq a(n) \\ l-(a(n)+1)
\geq 0 & \; l \geq a(n)+1 \end{cases}
\end{align*}
which imply
\begin{align*}
& 8 \cdot f'_k \cdot f^m \cdot q_n(f^2) \notin 4 \cdot \ZZ[\chi]/ I \langle 2n+2 \rangle,\\
& 8 \cdot f'_k \cdot f^m \cdot q_n(f^2) \cdot (1-\chi)^{2b(n)-1} \in
4 \cdot \ZZ[\chi]/ I \langle 2n+2 \rangle.
\end{align*}
For $b(n) = 0$ we have
\begin{align*}
& w_l \left( 8 \cdot f'_k \cdot f^m \cdot q_n(f^2) \right) - \left( 2+2n+2-l-2^{-l} \right) =\\
& \qquad \qquad  \begin{cases} \infty & \; l \leq a(n) \\
l-(a(n)+1)+\frac{1}{2^l} \geq 0 & \; l \geq a(n)+1 \end{cases}
\end{align*}
which implies
\[
8 \cdot f'_k \cdot f^m \cdot q_n(f^2) \in 4 \cdot \ZZ[\chi]/ I
\langle 2n+2 \rangle.
\]
From
\begin{align*}
& w_l \left( 8 \cdot f'_k \cdot f^m \cdot q_n(f^2) \cdot (1-\chi)^{2n} \right) - \left( 2+2n+3-l-2^{-l} \right) =\\
& \qquad \qquad \begin{cases} \infty & \; l \leq a(n) \\
-\frac{1}{2^{a(n)+1}} < 0 & \; l = a(n)+1 \\
l-(a(n)+2)+\frac{2^{a(n)+1}-1}{2^l} \geq 0 &  \; l \geq a(n)+2
\end{cases}
\end{align*}
we conclude
\begin{align*}
& 8 \cdot f'_k \cdot f^m \cdot q_n(f^2) \cdot (1-\chi)^{2n} \notin 4 \cdot \ZZ[\chi]/ I \langle 2n+3 \rangle \mathrm{\quad and \; hence}\\
& 8 \cdot f'_k \cdot f^m \cdot q_n(f^2) \notin 4 \cdot \ZZ[\chi]/ I
\langle 2n+3 \rangle.
\end{align*}
It remains to show
\[
8 \cdot f'_k \cdot f^m \cdot q_n(f^2) \cdot (1-\chi)^{2n+1+2^{a(n)}
(2^s-2)} \in 4 \cdot \ZZ[\chi]/ I \langle 2n+2+s \rangle \mathrm{\;
for \; all \;} s \geq 1.
\]
\begin{align*}
& w_l \left( 8 \cdot f'_k \cdot f^m \cdot q_n(f^2) \cdot (1-\chi)^{2n+1+2^{a(n)} (2^s-2)} \right) - \left( 2+2n+2+s-l-2^{-l} \right) =\\
& \qquad \qquad \begin{cases} \infty & \; l \leq a(n) \\
l-a(n)-s-1+2^{a(n)+s-l} &  \; l \geq a(n)+1 \end{cases}
\end{align*}
We set $c := a(n)+s-l$ and have to show $2^c \geq c + 1$ for all $c
\in \ZZ$. This is obviously true for $c \leq -1$. The statement for
$c \geq 0$ follows by induction.
\end{proof}

Notice that the polynomials $q_n$ have slightly better properties
when $b(n) = 0$. This suggests that there might exist better
polynomials than $q_n$ when $b(n) > 0$. This turns out to be true,
there exist polynomials $r_n^-$ of degree $n$ with leading
coefficient $1$ such that
\[
8 \cdot f'_k \cdot f^m \cdot r^-_n(f^2) \in 4 \cdot \ZZ[\chi]/ I
\langle 2n+2 \rangle.
\]
Their construction needs some preparation.

\begin{lem}\label{c_min=odd}
Let $g \in 2 \cdot \RhGm$ where $G=\ZZ_{2^K}$. If $g \notin 4 \cdot
\RhGm$ then there exists an odd natural number $c$ such that
\begin{align*}
&g(\chi) \cdot (1-\chi)^c \in 4 \cdot \ZZ[\chi]/ I \langle K \rangle,\\
&g(\chi) \cdot (1-\chi)^{c-1} \notin 4 \cdot \ZZ[\chi]/ I \langle K
\rangle.
\end{align*}
\end{lem}
\begin{proof}
The element $g \in 2 \cdot \RhGm$ can be written as
\[
g(\chi) = 2 \cdot \sum_{k=1}^{2^{K-1}} a_k \cdot (\chi^k -
\chi^{-k})
\]
with $a_k \in \ZZ$. We set
\[
\overline{g} (\chi) := \sum_{k=1}^{2^{K-1}} a_k \cdot (\chi^k -
\chi^{-k}) \in \ZZ_2[\chi] / I \langle K \rangle.
\]
Now suppose that $g \notin 4 \cdot \RhGm$ i.e. $h \neq 0$. Since
\[
\chi^k - \chi^{-k} = (\chi - \chi^{-1}) \cdot \left( \chi^{1-k} +
\chi^{3-k} + \ldots + \chi^{k-3} + \chi^{k-1} \right),
\]
any element $y \in \ZZ_2[\chi] / I \langle K \rangle$ of the shape
$y(\chi) = \sum_{k=1}^m c_k \cdot (\chi^k - \chi^{-k})$ can be
written as
\[
y(\chi) = \left( \chi - \chi^{-1} \right)\cdot \left( c'_0 +
\sum_{k=1}^{m-1} c'_k \cdot (\chi^k - \chi^{-k}) \right).
\]
Now, we transform $\overline{g}$ in this way and repeat the
transformation as long as the occurring $c'_0$ is zero. We finally
get
\[
\overline{g} (\chi) = \left( \chi - \chi^{-1} \right)^n \cdot \left(
1 + \sum_{k=1}^{2^{K-1}-n} b_k \cdot (\chi^k - \chi^{-k}) \right).
\]
We set $c := 2^K-2n-1$. Notice that we have in $\ZZ_2[\chi]$
\[
1+\chi+\ldots+\chi^{2^K-1} = \prod_{r=1}^{K-1} (1+\chi^{2^r}) =
\prod_{r=1}^{K-1} (1-\chi)^{2^r} = (1-\chi)^{2^K-1}
\]
and
\[
\left( \chi - \chi^{-1} \right)^n = \left( \chi^{-1} \cdot
(1-\chi)^2 \right)^n = \chi^{-n} \cdot (1-\chi)^{2n}.
\]
Therefore, we calculate in $\ZZ_2[\chi] / I \langle K \rangle$
\[
\left( \chi - \chi^{-1} \right)^n \cdot (1-\chi)^c = \chi^{-n}
(1-\chi)^{2^K-1} = 0.
\]
This implies $\overline{g} (\chi) \cdot (1-\chi)^c = 0$ in
$\ZZ_2[\chi] / I \langle K \rangle$ and hence
\[
g(\chi) \cdot (1-\chi)^c \in 4 \cdot \ZZ[\chi]/ I \langle K \rangle.
\]
It remains to show
\[
g(\chi) \cdot (1-\chi)^{c-1} \notin 4 \cdot \ZZ[\chi]/ I \langle K
\rangle.
\]
We prove this by contradiction. Suppose $g(\chi) \cdot
(1-\chi)^{c-1} \in 4 \cdot \ZZ[\chi]/ I \langle K \rangle$ which
implies $\overline{g} (\chi) \cdot (1-\chi)^{c-1} = 0$ in
$\ZZ_2[\chi] / I \langle K \rangle$. This means that there exists $q
\in \ZZ_2[\chi]$ with
\[
\left( \chi - \chi^{-1} \right)^n \cdot \left( 1 +
\sum_{k=1}^{2^{K-1}-n} b_k \cdot (\chi^k - \chi^{-k}) \right) \cdot
(1-\chi)^{c-1} = q(\chi) \cdot (1+\chi+\ldots+\chi^{2^K-1}).
\]
We conclude in $\ZZ_2[\chi]$
\[
\chi^{-n} \cdot (1-\chi)^{2n} \cdot \left( 1 +
\sum_{k=1}^{2^{K-1}-n} b_k \cdot (\chi^k - \chi^{-k}) \right) \cdot
(1-\chi)^{c-1} = q(\chi) \cdot (1-\chi)^{2^K-1}
\]
and hence
\[
\chi^{-n} \cdot \left( 1 + \sum_{k=1}^{2^{K-1}-n} b_k \cdot (\chi^k
- \chi^{-k}) \right) = q(\chi) \cdot (1-\chi).
\]
This implies the desired contradiction
\[
1 = 1^{-n} \cdot \left( 1 + \sum_{k=1}^{2^{K-1}-n} b_k \cdot (1^k -
1^{-k}) \right) = q(1) \cdot (1-1) = 0 \mathrm{\quad in \;} \ZZ_2.
\]
\end{proof}

\begin{lem}\label{c_min-add}
Let $g_i \in \QQ[\chi] / I \langle K \rangle$  for $i = 1,2$ and let
$c \geq 1$ be such that
\begin{align*}
& 2 \cdot g_i(\chi) \in 4 \cdot \ZZ[\chi] / I \langle K \rangle,\\
& g_i(\chi) \cdot (1-\chi)^c \in 4 \cdot \ZZ[\chi] / I \langle K \rangle,\\
& g_i(\chi) \cdot (1-\chi)^{c-1} \notin 4 \cdot \ZZ[\chi] / I
\langle K \rangle.
\end{align*}
Then
\[
\left( g_1(\chi) + g_2(\chi) \right) \cdot (1-\chi)^{c-1} \in 4
\cdot \ZZ[\chi] / I \langle K \rangle.
\]
\end{lem}
\begin{proof}
Since $2 \cdot g_i(\chi) \cdot (1-\chi)^{c-1} \in 4 \cdot \ZZ[\chi]
/ I \langle K \rangle$, there exist polynomials $h_i \in \ZZ[\chi]$
($i=1,2$) such that $2 \cdot g_i(\chi) \cdot (1-\chi)^{c-1}$ and $4
\cdot h_i(\chi)$ coincide modulo $I \langle K \rangle$. We can
require that $\deg(h_i) \leq 2^K-2$. Let $\overline{h_i}$ be the
image of $h_i$ under the epimorphism $\ZZ[\chi] \twoheadrightarrow
\ZZ_2[\chi]$. Notice that $1+\chi+\ldots+\chi^{2^K-1}$ divides
$\overline{h_i}(\chi) \cdot (1-\chi)$ in $\ZZ_2[\chi]$ because of
$g_i(\chi) \cdot (1-\chi)^c \in 4 \cdot \ZZ[\chi]/ I \langle K
\rangle$. In $\ZZ_2[\chi]$ we have
\[
1+\chi+\ldots+\chi^{2^K-1} = \prod_{r=0}^{K-1} (1+\chi^{2^r}) =
\prod_{r=0}^{K-1} (1+\chi)^{2^r} = (1+\chi)^{2^K-1}.
\]
Therefore, $(1+\chi)^{2^K-2}$ divides $\overline{h_i}(\chi)$. Since
$g_i(\chi) \cdot (1-\chi)^{c-1} \notin 4 \cdot \ZZ[\chi] / I \langle
K \rangle$, we have $\overline{h_i}(\chi) \neq 0$. We conclude from
$\deg(\overline{h_i}) \leq \deg(h_i) \leq 2^K-2$ that
$\overline{h_i}(\chi) = (1+\chi)^{2^K-2}$. Therefore,
\[
\overline{h_1}(\chi) + \overline{h_2}(\chi) = 2 \cdot
(1+\chi)^{2^K-2} = 0 \mathrm{\quad in \;} \ZZ_2[\chi].
\]
This implies $h_1(\chi) + h_2(\chi) \in 2 \cdot \ZZ[\chi]$. In
$\QQ[\chi] / I \langle K \rangle$ we finally conclude
\[
\left( g_1(\chi) + g_2(\chi) \right) \cdot (1-\chi)^{c-1} = 2 \cdot
\left( h_1(\chi) + h_2(\chi) \right) \in 4 \cdot \ZZ[\chi] / I
\langle K \rangle.
\]
\end{proof}

\begin{lem}\label{mult2}
Let $g \in \QQ[\chi]/ I \langle K+1 \rangle$ such that $\pr_l(g) \in
4 \cdot \ZZ[\chi] / \langle 1+\chi^{2^K} \rangle$. Then
\[
g \in 4 \cdot \ZZ[\chi]/ I \langle K \rangle \; \Longleftrightarrow
\; 2g \in 4 \cdot \ZZ[\chi]/ I \langle K+1 \rangle.
\]
\end{lem}
\begin{proof}
Assume first $g \in 4 \cdot \ZZ[\chi]/ I \langle K \rangle$. Let $h
\in \ZZ[\chi]$ be such that $4h$ and $g$ coincide in $\QQ[\chi] /
\langle 1+\chi^{2^K} \rangle$ and let $k \in \ZZ[\chi]$ such that
$4k$ and $g$ coincide in $\QQ[\chi]/ I \langle K \rangle$. Then we
obtain in $\QQ[\chi]/ I \langle K+1 \rangle$ the equation
\[
2 \cdot g(\chi) = 4 \cdot (1+\chi^{2^K}) \cdot k(\chi) + 4 \cdot
(1-\chi^{2^K}) \cdot h(\chi).
\]
which shows $2g \in 4 \cdot \ZZ[\chi]/ I \langle K+1 \rangle$.

Now assume $2g \in 4 \cdot \ZZ[\chi]/ I \langle K+1 \rangle$. We
want to show $g \in 4 \cdot \ZZ[\chi]/ I \langle K \rangle$. Let $h
\in \ZZ[\chi]$ be again such that $4h$ and $g$ coincide in
$\QQ[\chi] / \langle 1+\chi^{2^K} \rangle$ and let $k \in \ZZ[\chi]$
be such that $4k$ and $2g$ (resp. $2k$ and $g$) coincide in
$\QQ[\chi]/ I \langle K+1 \rangle$. Then $2 \cdot k(\chi)$ and
$k(\chi) \cdot (1+\chi^{2^K}) + 2 \cdot h(\chi) \cdot
(1-\chi^{2^K})$ coincide in $\QQ[\chi]/ I \langle K \rangle$ and in
$\QQ[\chi]/ \langle 1+\chi^{2^K} \rangle$ and hence also in
$\QQ[\chi]/ I \langle K+1 \rangle$. Therefore there exists an $r \in
\QQ[\chi]$ with
\[
2 \cdot k(\chi) = k(\chi) \cdot (1+\chi^{2^K}) + 2 \cdot h(\chi)
\cdot (1-\chi^{2^K}) + r(\chi) \cdot
(1+\chi+\ldots+\chi^{2^{K+1}-1}).
\]
We conclude $r \in \ZZ[\chi]$. Under the epimorphism $\ZZ[\chi]
\twoheadrightarrow \ZZ_2[\chi]$ we get
\[
0 = \overline{k}(\chi) \cdot (1+\chi^{2^K}) + \overline{r}(\chi)
\cdot (1+\chi+\ldots+\chi^{2^{K+1}-1})
\]
and hence
\[
0 = \overline{k}(\chi) + \overline{r}(\chi) \cdot
(1+\chi+\ldots+\chi^{2^K-1}).
\]
We set $s(\chi) := k(\chi) + r(\chi) \cdot
(1+\chi+\ldots+\chi^{2^K-1}) \in \ZZ[\chi]$. The vanishing of $s$
under the epimorphism $\ZZ[\chi] \twoheadrightarrow \ZZ_2[\chi]$
implies the existence of $t \in \ZZ[\chi]$ with $2t = s$. We
conclude in $\QQ[\chi]/ I \langle K \rangle$
\[
g = 2k = 2s = 4t.
\]
This shows $g \in 4 \cdot \ZZ[\chi]/ I \langle K \rangle$.
\end{proof}

The desired polynomials $r_n^-$ are obtained inductively. The
crucial inductive step is based on the following proposition. The
idea is motivated by the properties of $q_n$ when $b (n) = 0$ and is
based on the following observation: If we assume for a given $n \in
\NN$ with $b(n) > 0$ the existence of polynomials $\widetilde q_l$
for $l \leq \lfloor \frac{n}{2} \rfloor -1$ which are slightly
better than $q_l$ then we are able to conclude the existence of a
$\widetilde q_n$ which is also better than $q_n$.

\begin{prop}\label{r_n-prop}
Let $n \geq 0$, $k \geq 1$ and $m \in \{1,2\}$. Let $\wq_l \in
\ZZ[\chi]$ be polynomials for $0 \leq l \leq
\lfloor\frac{n}{2}\rfloor-1$ such that
\begin{align*}
& 8 \cdot f'_k \cdot f^m \cdot \wq_l(f^2) \in 4 \cdot \ZZ[\chi]/ I \langle 2l+2 \rangle,\\
& 8 \cdot f'_k \cdot f^m \cdot \wq_l(f^2) \cdot (1-\chi)^{2l} \notin 4 \cdot \ZZ[\chi]/ I \langle 2l+3 \rangle,\\
& 8 \cdot f'_k \cdot f^m \cdot \wq_l(f^2) \cdot
(1-\chi)^{2l+1+2^{a(l)} (2^s-2)} \in 4 \cdot \ZZ[\chi]/ I \langle
2l+2+s \rangle
\end{align*}
for all $0 \leq l \leq \lfloor\frac{n}{2}\rfloor-1$, $s \geq 1$.
Then there exist unique $a_l \in \{0,1\}$ for $0 \leq l \leq
\lfloor\frac{n}{2}\rfloor-1$ such that
\[
\wq_n := q_n + \sum_{l=0}^{\lfloor\frac{n}{2}\rfloor-1} a_l \cdot
2^{2(n-l)-1} \cdot \wq_l
\]
satisfies
\[
8 \cdot f'_k \cdot f^m \cdot \wq_n(f^2) \in 4 \cdot \ZZ[\chi]/ I
\langle 2n+2 \rangle.
\]
Moreover, we get
\begin{align*}
& 8 \cdot f'_k \cdot f^m \cdot \wq_n(f^2) \cdot (1-\chi)^{2n} \notin 4 \cdot \ZZ[\chi]/ I \langle 2n+3 \rangle,\\
& 8 \cdot f'_k \cdot f^m \cdot \wq_n(f^2) \cdot
(1-\chi)^{2n+1+2^{a(n)} (2^s-2)} \in 4 \cdot \ZZ[\chi]/ I \langle
2n+2+s \rangle
\end{align*}
for all $s \geq 1$.
\end{prop}
\begin{proof}
From Proposition \ref{K(q-n)-prop} we know that $b(n) = 0$ implies
\[
8 \cdot f'_k \cdot f^m \cdot q_n(f^2) \in 4 \cdot \ZZ[\chi]/ I
\langle 2n+2 \rangle.
\]
In the case $b(n)>0$ we have
\[
8 \cdot f'_k \cdot f^m \cdot q_n(f^2) \cdot (1-\chi)^{2b(n)-1} \in 4
\cdot \ZZ[\chi]/ I \langle 2n+2 \rangle.
\]
Now let $c \geq 0$ be the smallest number such that there exist
coefficients $a_l$ satisfying
\begin{equation}\label{eq_c}
8 \cdot f'_k \cdot f^m \cdot \left( q_n(f^2) +
\sum_{l=0}^{\lfloor\frac{n}{2}\rfloor-1} a_l \cdot 2^{2(n-l)-1}
\cdot \wq_l(f^2) \right) \cdot (1-\chi)^c \in 4 \cdot \ZZ[\chi]/ I
\langle 2n+2 \rangle.
\end{equation}
We have to show $c=0$. We will give a proof by contradiction and
assume that $c > 0$. We already know that $c \leq 2b(n)-1$. From
Lemma \ref{c_min=odd} we conclude that $c$ is odd. We set $l' :=
\frac{c-1}{2}$. Since $2l'+1 = c \leq 2b(n)-1 \leq b(n)+2^{a(n)}-2 =
n-1$, we conclude $l' \leq \lfloor\frac{n}{2}\rfloor-1$. Let $(a_l)$
be a choice of coefficients with the property (\ref{eq_c}). We set
\[
g_1(\chi) := 8 \cdot f'_k \cdot f^m \cdot q_n(f^2) + 8 \cdot f'_k
\cdot f^m \cdot \sum_{l=0}^{\lfloor\frac{n}{2}\rfloor-1} a_l \cdot
2^{2(n-l)-1} \cdot \wq_l(f^2).
\]
Notice that $2 \cdot g_1(\chi) \in 4 \cdot \ZZ[\chi]/ I \langle 2n+2
\rangle$ because all summands lie in this ring (use Lemma
\ref{mult2}). Moreover, we have $g_1(\chi) \cdot (1-\chi) \in 4
\cdot \ZZ[\chi]/ I \langle 2n+2 \rangle$ and $g_1(\chi) \notin 4
\cdot \ZZ[\chi]/ I \langle 2n+2 \rangle$. Define
\[
 g_2(\chi) := 8 \cdot f'_k \cdot f^m \cdot (-1)^{a_{l'}} \cdot 2^{2(n-l')-1} \cdot \wq_{l'}(f^2)
\]
Using Lemma \ref{mult2} we see that $2 \cdot g_2(\chi), \, g_2(\chi)
\cdot (1-\chi)^c \in 4 \cdot \ZZ[\chi]/ I \langle 2n+2 \rangle$ but
$g_2(\chi) \cdot (1-\chi)^{c-1} \notin 4 \cdot \ZZ[\chi]/ I \langle
2n+2 \rangle$. Now, we can use Lemma \ref{c_min-add} and get
\[
\left( g_1(\chi) + g_2(\chi) \right) \cdot (1-\chi)^{c-1} \in 4
\cdot \ZZ[\chi] / I \langle K \rangle.
\]
But this means that
\[
8 \cdot f'_k \cdot f^m \cdot \left( q_n(f^2) +
\sum_{l=0}^{\lfloor\frac{n}{2}\rfloor-1} a'_l \cdot 2^{2(n-l)-1}
\cdot \wq_l(f^2) \right) \cdot (1-\chi)^{c-1} \in 4 \cdot \ZZ[\chi]/
I \langle 2n+2 \rangle.
\]
where the coefficients $(a'_l)$ are given by
\[
a'_l :=
\begin{cases} a_l & \; l \neq l' \\ a_{l'} + (-1)^{a_{l'}} & \; l = l' \end{cases}.
\]
This is a contradiction to the minimality of $c$. So far we have
shown that there exist $a_l \in \{0,1\}$ $(0 \leq l \leq
\lfloor\frac{n}{2}\rfloor-1)$ such that
\[
\wq_n := q_n + \sum_{l=0}^{\lfloor\frac{n}{2}\rfloor-1} a_l \cdot
2^{2(n-l)-1} \cdot \wq_l
\]
satisfies
\[
8 \cdot f'_k \cdot f^m \cdot \wq_n(f^2) \in 4 \cdot \ZZ[\chi]/ I
\langle 2n+2 \rangle.
\]

Our next aim is to show the uniqueness of the coefficients. We will
give a proof by contradiction. Assume that there exist two different
choices of coefficients $(a_l)$, $(a'_l)$ such that the
corresponding $\wq_n$, $\wq'_n$ satisfy
\[
8 \cdot f'_k \cdot f^m \cdot \wq_n(f^2), \, 8 \cdot f'_k \cdot f^m
\cdot \wq'_n(f^2) \in 4 \cdot \ZZ[\chi]/ I \langle 2n+2 \rangle.
\]
We set $b_l := a_l - a'_l$ and conclude
\[
8 \cdot f'_k \cdot f^m \cdot \left(
\sum_{l=0}^{\lfloor\frac{n}{2}\rfloor-1} b_l \cdot 2^{2(n-l)-1}
\cdot \wq_l(f^2) \right) \in 4 \cdot \ZZ[\chi]/ I \langle 2n+2
\rangle.
\]
Let ${\hat l}$ be the largest element with $b_{\hat l} \neq 0$ (i.e.
$b_{\hat l} = \pm 1$). Using Lemma \ref{mult2}, we conclude
\[
8 \cdot f'_k \cdot f^m \cdot 2^{2(n-l)-1} \cdot \wq_l(f^2) \cdot
(1-\chi)^{2{\hat l}} \in 4 \cdot \ZZ[\chi]/ I \langle 2n+2 \rangle
\mathrm{ \quad for \;} l < {\hat l}
\]
and
\[
8 \cdot f'_k \cdot f^m \cdot 2^{2(n-l)-1} \cdot \wq_{\hat l}(f^2)
\cdot (1-\chi)^{2{\hat l}} \notin 4 \cdot \ZZ[\chi]/ I \langle 2n+2
\rangle.
\]
This implies
\[
8 \cdot f'_k \cdot f^m \cdot \left(
\sum_{l=0}^{\lfloor\frac{n}{2}\rfloor-1} b_l \cdot 2^{2(n-l)-1}
\cdot \wq_l(f^2) \right) \cdot (1-\chi)^{2{\hat l}} \notin 4 \cdot
\ZZ[\chi]/ I \langle 2n+2 \rangle
\]
contradicting
\[
8 \cdot f'_k \cdot f^m \cdot \left(
\sum_{l=0}^{\lfloor\frac{n}{2}\rfloor-1} b_l \cdot 2^{2(n-l)-1}
\cdot \wq_l(f^2) \right) \in 4 \cdot \ZZ[\chi]/ I \langle 2n+2
\rangle.
\]
It remains to prove
\begin{align*}
& 8 \cdot f'_k \cdot f^m \cdot \wq_n(f^2) \cdot (1-\chi)^{2n} \notin 4 \cdot \ZZ[\chi]/ I \langle 2n+3 \rangle,\\
& 8 \cdot f'_k \cdot f^m \cdot \wq_n(f^2) \cdot
(1-\chi)^{2n+1+2^{a(n)} (2^s-2)} \in 4 \cdot \ZZ[\chi]/ I \langle
2n+2+s \rangle
\end{align*}
for all $s \geq 1$. Using Lemma \ref{mult2} we obtain
\[
8 \cdot f'_k \cdot f^m \cdot 2^{2(n-l)-1} \cdot \wq_l(f^2) \cdot
(1-\chi)^{2l+1+2^{a(l)} (2^{s'}-2)} \in 4 \cdot \ZZ[\chi]/ I \langle
2n+1+s' \rangle
\]
for all $s' \geq 2$. Let $l \leq \lfloor\frac{n}{2}\rfloor-1$.
Setting $s' := s+1$ we conclude
\[
8 \cdot f'_k \cdot f^m \cdot 2^{2(n-l)-1} \cdot \wq_l(f^2) \cdot
(1-\chi)^{2n+1+2^{a(n)} (2^s-2)} \in 4 \cdot \ZZ[\chi]/ I \langle
2n+2+s \rangle
\]
because
\begin{align*}
& 2l+1+2^{a(l)} (2^{s+1}-2) \; \leq \; 2(\frac{n}{2}-1)+1+2^{a(n)-1} (2^{s+1}-2) \; =\\
& 2n+1+2^{a(n)} (2^s-2) - b(n) \; \leq \; 2n+1+2^{a(n)} (2^s-2).
\end{align*}
Setting $s' := 2$ we obtain
\[
8 \cdot f'_k \cdot f^m \cdot 2^{2(n-l)-1} \cdot \wq_l(f^2) \cdot
(1-\chi)^{2n} \in 4 \cdot \ZZ[\chi]/ I \langle 2n+3 \rangle
\]
because
\[
2l+1+2^{a(l)} (2^2-2) \leq 2(\frac{n}{2}-1)+1+2^{a(n)} = 2n-b(n)
\leq 2n.
\]
Therefore, it suffices to show
\begin{align*}
& 8 \cdot f'_k \cdot f^m \cdot q_n(f^2) \cdot (1-\chi)^{2n} \notin 4 \cdot \ZZ[\chi]/ I \langle 2n+3 \rangle,\\
& 8 \cdot f'_k \cdot f^m \cdot q_n(f^2) \cdot
(1-\chi)^{2n+1+2^{a(n)} (2^s-2)} \in 4 \cdot \ZZ[\chi]/ I \langle
2n+2+s \rangle
\end{align*}
for all $s \geq 1$. But this was proved in Proposition
\ref{K(q-n)-prop}.
\end{proof}

\noindent Notice that the assumptions in Proposition \ref{r_n-prop}
are trivially fulfilled if $n=0,1$.

\begin{defn} \label{r_n-defn}
We define $r^-_n \in \ZZ[\chi]$ as the polynomials $\wq_n$ we obtain
successively from Proposition \ref{r_n-prop} starting with $n = 0$
and proceeding with $n = 1, 2, 3, \ldots$.
\end{defn}

\noindent For example, $r^-_0 = q_0$, $r^-_1 = q_1$, $r^-_2 = q_2 +
2^3 \cdot q_0$, $r^-_3 = q_3$, $r^-_4 = q_4 + 2^7 \cdot q_0$.

\begin{cor} \label{r_n-properties}
The polynomial $r^-_n$ is of degree $n \in \NN$ with leading
coefficient $1$ and it satisfies
\begin{align*}
& 8 \cdot f'_k \cdot f^m \cdot r^-_n(f^2) \in 4 \cdot \ZZ[\chi]/ I \langle 2n+2 \rangle,\\
& 8 \cdot f'_k \cdot f^m \cdot r^-_n(f^2) \cdot (1-\chi)^{2n} \notin 4 \cdot \ZZ[\chi]/ I \langle 2n+3 \rangle,\\
& 8 \cdot f'_k \cdot f^m \cdot r^-_n(f^2) \cdot
(1-\chi)^{2n+1+2^{a(n)} (2^s-2)} \in 4 \cdot \ZZ[\chi]/ I \langle
2n+2+s \rangle
\end{align*}
for all $s \geq 1$.
\end{cor}

Are the polynomials $r^-_n$ best possible? Or does there exist a
polynomial $q$ of degree $n$ with leading coefficient $1$ such that
$8 \cdot f_k \cdot q(f^2) \in \ZZ[\chi]/ I \langle 2n+3 \rangle$? In
the next section we will see that any polynomial $q$ of degree $n$
with the property $8 \cdot f_k \cdot q(f^2) \in \ZZ[\chi]/ I \langle
2n+3 \rangle$ is of the shape
\[
\sum_{l=0}^n a_l \cdot 2^{\max\{2(n-l)+1,0\}} \cdot r^-_l
\]
with $a_n \in \ZZ$. Hence such a polynomial can not have $1$ as
leading coefficient.

\

\subsection{The equation $\mathbf{A^k_K(d) = B_K(d)}$} \label{subsec:A=B}

\

\medskip

In this subsection we first prove $A^k_K(d) = B_K(d)$ for $d= 2e+1$.
Recall that
\begin{align*}
& A^k_K(2e+1) := \left\lbrace q \in \ZZ[x] \mid \deg(q) \leq e-1, \; 8 \cdot f_k \cdot q(f^2) \in 4 \cdot \ZZ[\chi]/ I \langle K \rangle \right\rbrace,\\
& B_K(2e+1) := \left\lbrace \sum_{n=0}^{e-1} a_n \cdot
2^{\max\{K-2n-2,0\}} \cdot r^-_n \mid a_n \in \ZZ \right\rbrace.
\end{align*}

We want to consider a slightly more general situation and prove
$A^{k,m}_K(2e+1) = B_K(2e+1)$ where $A^{k,m}_K(2e+1)$ is defined as
follows.

\begin{defn}
Let $K, k \geq 1$, $e \geq 2$, $m \in \{1,2\}$. Define
\[
A^{k,m}_K(2e+1) := \left\lbrace q \in \ZZ[x] \mid \deg(q) \leq e-1,
\; 8 \cdot f'_k \cdot f^m \cdot q(f^2) \in 4 \cdot \ZZ[\chi]/ I
\langle K \rangle \right\rbrace
\]
\end{defn}

Notice that $A^{k,1}_K(2e+1) = A^k_K(2e+1)$.

\begin{thm}\label{A=B-odd}
Let $K, k \geq 1$, $e \geq 2$, $m \in \{1,2\}$. Then
$A^{k,m}_K(2e+1) = B_K(2e+1)$. In particular,
\[
A^k_K(2e+1) = B_K(2e+1).
\]
\end{thm}
\begin{proof}
Since $8 \cdot f'_k \cdot f^m \cdot r^-_n(f^2) \in 4 \cdot
\ZZ[\chi]/ I \langle 2n+2 \rangle$, Lemma \ref{mult2} implies
\[
8 \cdot f'_k \cdot f^m \cdot 2^{\max\{K-2n-2,0\}} \cdot r^-_n(f^2)
\in 4 \cdot \ZZ[\chi]/ I \langle K \rangle.
\]
This proves $A^{k,m}_K(2e+1) \supseteq B_K(2e+1)$.

It remains to show $A^{k,m}_K(2e+1) \subseteq B_K(2e+1)$. We will
give a proof by induction with respect to $K$. For the basis case
$K=1$ we get
\[
A^{k,m}_1(2e+1) = \left\lbrace q \in \ZZ[x] \mid \deg(q) \leq e-1
\right\rbrace = B_1(2e+1).
\]
Inductive step: We assume that $A^{k,m}_{K-1}(2e+1) \subseteq
B_{K-1}(2e+1)$ ($K \geq 2$) and have to prove $A^{k,m}_K(2e+1)
\subseteq B_K(2e+1)$. Let $q \in A^{k,m}_K(2e+1)$. Since
\[
A^{k,m}_K(2e+1) \subseteq A^{k,m}_{K-1}(2e+1) \subseteq
B_{K-1}(2e+1),
\]
we can write $q$ as $q = \sum_{n=0}^{e-1} a_n \cdot
2^{\max\{K-2n-3,0\}} \cdot r^-_n$ with $a_n \in \ZZ$. The polynomial
$q$ lies in $B_K(2e+1)$ if $a_n$ is even for all $n$ with $2n+2 \leq
K-1$. We set
\[
M := \left\lbrace 0 \leq n \leq e-1 \mid 2n+2 \leq K-1, \mathrm{\;}
a_n \mathrm{\; is \; odd} \right\rbrace.
\]
It remains to show $M = \emptyset$. We will give a proof by
contradiction and assume $M \neq \emptyset$. Since $q \in
A^{k,m}_K(2e+1)$ and
\begin{align*}
& \sum_{n \notin M} a_n \cdot 2^{\max\{K-2n-3,0\}} \cdot r^-_n + \sum_{n \in M} (a_n-1) \cdot 2^{\max\{K-2n-3,0\}} \cdot r^-_n\\
& \qquad \qquad \in B_K(2e+1) \subseteq A^{k,m}_K(2e+1),
\end{align*}
we have
\[
\sum_{n \in M} 2^{K-2n-3} \cdot r^-_n \in A^{k,m}_K(2e+1).
\]
This implies
\begin{equation}\label{A=B-eq1}
\sum_{n \in M} 2^{K-2n} \cdot f'_k \cdot f^m \cdot r^-_n(f^2) \cdot
(1-\chi)^{2 \cdot \max(M)} \in 4 \cdot \ZZ[\chi]/ I \langle K
\rangle.
\end{equation}
Using Lemma \ref{mult2} we conclude from
\[
8 \cdot f'_k \cdot f^m \cdot r^-_n(f^2) \cdot (1-\chi)^{2n+1} \in 4
\cdot \ZZ[\chi]/ I \langle 2n+3 \rangle
\]
that
\[
2^{K-2n} \cdot f'_k \cdot f^m \cdot r^-_n(f^2) \cdot (1-\chi)^{2n+1}
\in 4 \cdot \ZZ[\chi]/ I \langle K \rangle
\]
and hence, if $n < \max (M)$ then
\begin{equation}\label{A=B-eq2}
2^{K-2n} \cdot f'_k \cdot f^m \cdot r^-_n(f^2) \cdot (1-\chi)^{2
\cdot \max(M)} \in 4 \cdot \ZZ[\chi]/ I \langle K \rangle.
\end{equation}
The property
\[
8 \cdot f'_k \cdot f^m \cdot r^-_{\max(M)}(f^2) \cdot (1-\chi)^{2
\cdot \max(M)} \notin 4 \cdot \ZZ[\chi]/ I \langle 2 \cdot \max(M) +
3 \rangle
\]
and Lemma \ref{mult2} imply
\begin{equation}\label{A=B-eq3}
2^{K-2 \cdot \max(M)} \cdot f'_k \cdot f^m \cdot r^-_{\max(M)}(f^2)
\cdot (1-\chi)^{2 \cdot \max(M)} \notin 4 \cdot \ZZ[\chi]/ I \langle
K \rangle.
\end{equation}
Combining (\ref{A=B-eq1}), (\ref{A=B-eq2}) and (\ref{A=B-eq3}) we
obtain the desired contradiction.
\end{proof}

We now come to the case $d=2e$.

\begin{defn}
Define $\beta: \ZZ[x] \to \ZZ[x]$ by
\[
\beta(q)(x) := \frac{x \cdot q(x) - q(1)}{x-1}
\]
and set
\[
r^+_n := \beta(r^-_n) \mathrm{\; for \;} n \geq 0.
\]
\end{defn}

Notice that $\beta$ is an isomorphism of $\ZZ$-modules and preserves
the degree of the polynomial. The inverse is given by
\[
\beta^{-1}(q)(x) = \frac{(x-1) \cdot q(x) + q(0)}{x}.
\]
$r^+_n$ is a polynomial of degree $n$ with leading coefficient $1$.

\begin{thm}\label{A=B-even}
Let $K, k \geq 1$, $e \geq 3$. Then
\[
A^k_K(2e) = B_K(2e).
\]
\end{thm}
\begin{proof}
Recall that
\begin{align*}
& A^k_K(2e) := \left\lbrace q \in \ZZ[x] \mid \deg(q) \leq e-2, \; 8 \cdot f'_k \cdot (f^2-1) \cdot q(f^2) \in 4 \cdot \ZZ[\chi]/ I \langle K \rangle \right\rbrace,\\
& B_K(2e) := \left\lbrace \sum_{n=0}^{e-2} a_n \cdot
2^{\max\{K-2n-2,0\}} \cdot r^+_n \mid a_n \in \ZZ \right\rbrace.
\end{align*}
For $q \in \ZZ[x]$ with $\deg(q) \leq e-2$ we conclude
\begin{align*}
& q \in A^k_K(2e) \Longleftrightarrow\\
& \Longleftrightarrow 8 \cdot f'_k \cdot (f^2-1) \cdot q(f^2) \in 4 \cdot \ZZ[\chi]/ I \langle K \rangle\\
& \Longleftrightarrow 8 \cdot f'_k \cdot \left((f^2-1) \cdot q(f^2) + q(0) \right) \in 4 \cdot \ZZ[\chi]/ I \langle K \rangle \quad (\mathrm{since} \; f'_k \in \RhG)\\
& \Longleftrightarrow 8 \cdot f'_k \cdot f^2 \cdot \beta^{-1}(q)(f^2) \in 4 \cdot \ZZ[\chi]/ I \langle K \rangle\\
& \Longleftrightarrow \beta^{-1}(q) \in A^{k,2}_K(2e-1)\\
& \Longleftrightarrow \beta^{-1}(q) \in B_K(2e-1) \quad (\mathrm{see \; Theorem \; \ref{A=B-odd}})\\
& \Longleftrightarrow q \in B_K(2e).
\end{align*}
\end{proof}

\small
\bibliography{lens-spaces}

\def\cprime{$'$}
\begin{thebibliography}{LdM71}

\bibitem[AS68]{Atiyah-Singer-III(1968)}
M.~F. Atiyah and I.~M. Singer.
\newblock The index of elliptic operators. {III}.
\newblock {\em Ann. of Math. (2)}, 87:546--604, 1968.

\bibitem[CF64]{Conner-Floyd(1964)}
P.~E. Conner and E.~E. Floyd.
\newblock {\em Differentiable periodic maps}.
\newblock Ergebnisse der Mathematik und ihrer Grenzgebiete, N. F., Band 33.
  Academic Press Inc., Publishers, New York, 1964.

\bibitem[HT00]{Hambleton-Taylor(2000)}
Ian Hambleton and Laurence~R. Taylor.
\newblock A guide to the calculation of the surgery obstruction groups for
  finite groups.
\newblock In {\em Surveys on surgery theory, Vol. 1}, volume 145 of {\em Ann.
  of Math. Stud.}, pages 225--274. Princeton Univ. Press, Princeton, NJ, 2000.

\bibitem[LdM71]{LdM(1971)}
S.~L{\'o}pez~de Medrano.
\newblock {\em Involutions on manifolds}.
\newblock Springer-Verlag, New York, 1971.
\newblock Ergebnisse der Mathematik und ihrer Grenzgebiete, Band 59.

\bibitem[L{\"u}c02]{Lueck(2001)}
Wolfgang L{\"u}ck.
\newblock A basic introduction to surgery theory.
\newblock In {\em Topology of high-dimensional manifolds, No. 1, 2 (Trieste,
  2001)}, volume~9 of {\em ICTP Lect. Notes}, pages 1--224. Abdus Salam Int.
  Cent. Theoret. Phys., Trieste, 2002.

\bibitem[Mil66]{Milnor(1966)}
J.~Milnor.
\newblock Whitehead torsion.
\newblock {\em Bull. Amer. Math. Soc.}, 72:358--426, 1966.

\bibitem[MM79]{Madsen-Milgram(1979)}
Ib~Madsen and R.~James Milgram.
\newblock {\em The classifying spaces for surgery and cobordism of manifolds},
  volume~92 of {\em Annals of Mathematics Studies}.
\newblock Princeton University Press, Princeton, N.J., 1979.

\bibitem[MS74]{Morgan-Sullivan(1974)}
John~W. Morgan and Dennis~P. Sullivan.
\newblock The transversality characteristic class and linking cycles in surgery
  theory.
\newblock {\em Ann. of Math. (2)}, 99:463--544, 1974.

\bibitem[Pet70]{Petrie(1970)}
Ted Petrie.
\newblock The {A}tiyah-{S}inger invariant, the {W}all groups {$L\sb{n}(\pi
  ,\,1)$}, and the function {$(te\sp{x}+1)/(te\sp{x}-1)$}.
\newblock {\em Ann. of Math. (2)}, 92:174--187, 1970.

\bibitem[Ran81]{Ranicki(1981)}
Andrew Ranicki.
\newblock {\em Exact sequences in the algebraic theory of surgery}, volume~26
  of {\em Mathematical Notes}.
\newblock Princeton University Press, Princeton, N.J., 1981.

\bibitem[Ran92]{Ranicki(1992)}
A.~A. Ranicki.
\newblock {\em Algebraic {$L$}-theory and topological manifolds}, volume 102 of
  {\em Cambridge Tracts in Mathematics}.
\newblock Cambridge University Press, Cambridge, 1992.

\bibitem[Ran08]{Ranicki(2008)}
Andrew Ranicki.
\newblock A composition formula for manifold structures.
\newblock {\em Preprint, arXiv: math.AT.0608705}, 2008.

\bibitem[Wal99]{Wall(1999)}
C.~T.~C. Wall.
\newblock {\em Surgery on compact manifolds}, volume~69 of {\em Mathematical
  Surveys and Monographs}.
\newblock American Mathematical Society, Providence, RI, second edition, 1999.
\newblock Edited and with a foreword by A. A. Ranicki.

\bibitem[You98]{Young(1998)}
Carmen~M. Young.
\newblock Normal invariants of lens spaces.
\newblock {\em Canad. Math. Bull.}, 41(3):374--384, 1998.

\end{thebibliography}
\bibliographystyle{alpha}

\end{document}